\documentclass[12pt]{article}
\usepackage{amssymb}
\usepackage{mathrsfs}
\usepackage{bbm}
\usepackage{amsmath}
\usepackage{multirow}
\usepackage{cases} %
\usepackage{graphicx}
\usepackage{color}
\renewcommand{\Box}{\framebox{\rule{0.3em}{0.0em}}}

\newtheorem{thm}{Theorem}[section]

\newtheorem{proposition}{Proposition}[section]
\newtheorem{example}{Example}[section]

\newtheorem{defn}{Definition}[section]

\newtheorem{corollary}{Corollary}[section]

\newcommand{\bgeqn}{\begin{eqnarray}}
\newcommand{\edeqn}{\end{eqnarray}}

\newcommand{\bgeq}{\begin{eqnarray*}}
\newcommand{\edeq}{\end{eqnarray*}}
\newcommand{\bgc}{\begin{center}}
\newcommand{\edc}{\end{center}}

\renewcommand{\Box}{\hfill \rule{2.3mm}{2.3mm}}

\makeatletter

\@addtoreset{equation}{section} \makeatother
\newenvironment{proof}{\noindent{\bf Proof. }}{\hfill $\Box$\medskip}

\title{Necessary optimality conditions for  implicit control systems with applications to control of differential algebraic equations}
\author{An Li\thanks{\baselineskip 9pt School of Mathematical Sciences, Xiamen University, Xiamen 361005, Fujian, China. The research of this author was partially
supported by the National Natural Science Foundation of China (Grant No. 11671335), the Natural Science Foundation of Fujian Province, China (Grant No. 2016J01033) and the Fundamental Research Funds for the Central Universities (Grant No. 20720160036).}
\ \ and \ \
Jane J. Ye\thanks{Corresponding author. Department of Mathematics and Statistics, University of Victoria, Victoria, B.C., Canada V8W 2Y2,
e-mail: janeye@uvic.ca.  The research of this author was  supported by NSERC.}\\
{\small Dedicated to the memory of Jonathan Michael Borwein}
}
\date{}
\begin{document}
\maketitle

\baselineskip 16pt

\vspace{4pt}{\bf Abstract.} In this paper we derive necessary optimality conditions for optimal control problems with nonlinear and nonsmooth implicit  control systems.
Implicit control systems have wide applications including differential algebraic equations (DAEs). The challenge in the study of implicit control system lies in that the system may be truly implicit, i.e., the Jacobian matrix of the constraint mapping may be singular.
Our necessary optimality conditions hold  under the so-called weak basic constraint qualification plus the calmness of a perturbed constraint mapping.  Such constraint qualifications allow for singularity of the Jacobian and hence is suitable for implicit systems.
Specifying these results to control of semi-explicit DAEs we obtain necessary optimality conditions for control of semi-explicit DAEs with index higher than one.

\vspace{4pt}{\bf Key Words}
Necessary optimality conditions, Optimal control, Implicit control systems, Differential algebraic equations, Calmness, Variational analysis

\vspace{4pt}{\bf AMS subject classification: 45K15, 49K21,49J53}

\medskip

\baselineskip 16pt
\parskip 2pt

%The first author's work was partially supported  by the National Natural Science Foundation of China (Grant No. 11101342).
%The second author's work was  partially supported by NSERC.

%\title{Enhanced
% Karush-Kuhn-Tucker
% condition  for mathematical programs with equilibrium constraints\thanks{The first
%author's work was supported in part  by NSERC. The second author's work was sponsored in part by China Scholarship Council and in part by NSERC.}}
%\author{Jane J. Ye\thanks{\baselineskip 9pt Department of Mathematics and Statistics, University of Victoria, Victoria, BC, V8W 3R4 Canada. E-mail: janeye@uvic.ca}\ \ and \ \
%Jin Zhang\thanks{\baselineskip 9pt Department of Mathematics and Statistics, University of Victoria, Victoria, BC, V8W 3R4 Canada. E-mail: zhangjin@uvic.ca. }}

\maketitle

\baselineskip 16pt

\medskip

\baselineskip 16pt
\parskip 2pt

\newpage

\section{Introduction.}
Given a time interval $[t_0,t_1]\subseteq \mathbb{R}$, very often, the dynamic behavior of a system is most naturally modeled as an implicit control system:
\begin{eqnarray*}
{\rm ICS}~~~~~~~~~~
\begin{array}{l}
 \varphi (x(t), u(t),\dot{x}(t))\in K_{\varphi}\,\quad a.e. \,t\in [t_0,t_1],\\
% \psi (x(t), u(t))\in K_{\psi}\quad a.e. \,t\in [t_0,t_1], \nonumber\\
 u(t) \in U \quad a.e. \,t\in [t_0,t_1],\nonumber\\
 (x(t_0), x(t_1))\in S,\nonumber
 \end{array}
\end{eqnarray*}
where
$\varphi :\mathbb{R}^{n_{x}}\times \mathbb{R}^{n_{u}}\times \mathbb{R}^{n_{x}}\rightarrow \mathbb{R}^{m}$,
%$\psi:\mathbb{R}^{n_{x}}\times \mathbb{R}^{n_{u}}\rightarrow  \mathbb{R}^d$;
$K_{\varphi}\subseteq \mathbb{R}^m$,
%$K_{\psi}\subseteq\mathbb{R}^d$,
$U\subseteq \mathbb{R}^{n_{u}}$, $S\subseteq\mathbb{R}^{n_{x}}\times \mathbb{R}^{n_{x}}$.

A particular case of the implicit control system  is
described by scalar equations, namely, differential algebraic equations (DAEs):
\begin{eqnarray*}
{\rm DAE}~~~~~~~~~~
\begin{array}{l} \varphi (x(t), u(t),\dot{x}(t))=0\,\quad a.e. \,t\in [t_0,t_1],\\
%\psi (x(t), u(t))=0\quad a.e. \,t\in [t_0,t_1], \nonumber\\
 u(t) \in U \quad a.e. \,t\in [t_0,t_1],\nonumber\\
 (x(t_0), x(t_1))\in S.\nonumber
 \end{array}
\end{eqnarray*}
A very popular model of a DAE is the so-called semi-explicit DAE:
\begin{eqnarray*}
 {\rm seDAE}~~~~~~~~~~
\begin{array}{l} \dot{x}(t)=\phi(x(t),y(t), u(t))\,\quad a.e. \,t\in [t_0,t_1],\\
 0=h(x(t),y(t), u(t))\quad a.e. \,t\in [t_0,t_1],\nonumber\\
 u(t) \in U \quad a.e. \,t\in [t_0,t_1],\nonumber\\
 (x(t_0), x(t_1))\in S,
\end{array}
\end{eqnarray*}
where
$\phi :\mathbb{R}^{n_{x}}\times \mathbb{R}^{n_{y}}\times \mathbb{R}^{n_{u}}\rightarrow \mathbb{R}^{n_{x}}$,
$h :\mathbb{R}^{n_{x}}\times \mathbb{R}^{n_{y}}\times \mathbb{R}^{n_{u}}\rightarrow \mathbb{R}^{n_y}$.

In the past couple decades, DAEs have become a very
important generalization of ordinary differential equations (ODEs) and have numerous
applications in mathematical modeling of various dynamical processes; see e.g. \cite{pet,BCM,Gerdts,roub} and the references therein.

In this paper we  study  the optimal control problem of an  implicit  system:
\begin{eqnarray*}
(P_{ICS})~~~~~~\min &&  J(x,u):=\int_{t_0}^{t_1} F( x(t), u(t),\dot{x}(t)) dt + f(x(t_0),x(t_1))\nonumber \\
s.t. && \varphi(x(t), u(t),\dot{x}(t))\in K_{\varphi}\,\quad a.e. \,t\in [t_0,t_1],\\
%&& \psi(x(t), u(t))\in K_{\psi}\quad a.e. \,t\in [t_0,t_1], \nonumber\\
&& u(t) \in U \quad a.e. \,t\in [t_0,t_1],\nonumber\\
&& (x(t_0), x(t_1))\in S,\nonumber
\end{eqnarray*}
where $F:\mathbb{R}^{n_{x}}\times \mathbb{R}^{n_{u}}\times \mathbb{R}^{n_{x}}\rightarrow \mathbb{R}$, $f:\mathbb{R}^{n_{x}}\times \mathbb{R}^{n_{x}}\rightarrow \mathbb{R}$. Our basic assumptions for problem $(P_{ICS})$ are very general.  We assume all sets involved are closed and all functions involved are locally Lipschitz continuous.

%If we introduce a vector variable $v(t):=\dot{x}(t)$, then we can  tranform $(P_{ICS})$ into the following equivalent form:
%\begin{eqnarray*}
%(P_{ECS})~~~~~~\min &&  J(x,u):=\int_{t_0}^{t_1} F( x(t), u(t),\dot{x}(t)) dt + f(x(t_0),x(t_1)),\nonumber \\
%s.t. && \dot{x}(t)=v(t)\,\quad a.e. \,t\in [t_0,t_1],\nonumber\\
%&& \varphi(x(t), u(t),v(t))\in K_{\varphi}  \quad a.e. \,t\in [t_0,t_1], \nonumber\\
%&&\psi(x(t), u(t))\in K_{\psi} \quad a.e. \,t\in [t_0,t_1], \label{mconstraint}\\
%&& u(t) \in U \quad a.e. \,t\in [t_0,t_1],\nonumber\\
%&& (x(t_0), x(t_1))\in S.\nonumber
%\end{eqnarray*}
%Our approach is to obtain a necessary optimality condition for problem $(P_{ECS})$ and then transform back to the one for the original problem $(P_{ICS})$.
%Problem $(P_{ECS})$ belongs to the class of  an optimal control problem with mixed state and control constraints. Recently it has been studied in  Clarke and de Pinho \cite[Theorem 4.3]{cp} under the so-called calibrated constraint qualification (CCQ). CCQ, however, is stronger than the Mangasarian-Fromovitz condition (MFC).  In the special case of the DAE, MFC becomes the linear independence condition which is very strong.

To our knowledge, there is very little   done for implicit control problems stated in such a general form as in ${( P_{ICS})}$.   In \cite[Theorem 1.1]{ledyaev}, for problem $(P_{ICS})$ with free end point,   Devdariani and Ledyaev  derived a  necessary optimality condition in a form that closely resembles the classical Pontryagin  maximum principle  with an implicitly defined Hamiltonian.
For control of semi-explicit DAEs, de Pinho and Vinter \cite{pinhovinter} derived a strong maximum principle under the assumption that the velocity set is convex and a weak maximum principle without the convexity assumption. Moreover a counter example in \cite{pinhovinter} shows that the strong maximum principle may not hold if  the velocity set is nonconvex.  The assumption on the convexity of the velocity set in  \cite[Theorem 3.1]{pinhovinter} was relaxed for the Bolza problem in \cite{roub}. A key assumption for the maximum principles in \cite{pinhovinter} to hold is   that the Jacobian matrix $\nabla_y h$ must be  nonsingular along the optimal pair. This means that the maximum principles derived in  \cite{pinhovinter} can  only be applied to  control of seDAEs with  index one. Recently some necessary optimality conditions for control of DAEs with higher indexes have been derived \cite{roub,KunkelM,Gerdts}.

%\textcolor{blue}{An aspect of some interest in \cite{cp} is the necessary conditions for optimal control of DAEs  in  \cite[Theorem 6.1]{cp}, which was extended by
%De Pinho \cite{pinho*} to cover problems with set constrained implicit control systems.
%But the author in \cite{pinho*} deals with the strong local minimizers and turn from the nonsmooth problems to smooth cases.

 In this paper, we  aim at deriving necessary optimality conditions for a (weak) local minimum of radius $R(\cdot)$ for nonsmooth problems $(P_{ICS})$ in the following sense.
A control or control function $u(\cdot)$ is a measurable function on $[t_0,t_1]$ such that $u(t)\in U$ for almost every $t \in [t_0,t_1]$. The state or state trajectory, corresponding to a given control $u(\cdot)$, refers to an absolutely continuous function  $x(\cdot)$ which together with $u(\cdot)$ satisfying all conditions in (ICS). We call such a pair $(x(\cdot), u(\cdot))$ an admissible pair.  For simplicity we may omit the time variable and write $x, u$ instead of $ x(\cdot), u(\cdot)$,  respectively. Let {$R(t):[t_0,t_1]\rightarrow (0,+\infty]$ }be a radius function. We say
 that $(x_{*}, u_{*})$ is a local minimum of radius $R(\cdot)$ for  ${( }P_{ICS} {)}$ if 	 $(x_{*}, u_{*})$
minimizes the value of the cost
function $J(x, u)$ over all admissible pairs $(x,u)$ which satisfies
\begin{eqnarray}
&& | x(t)-x_{*}(t)|\leq\varepsilon \mbox{ a.e.} \,t\in [t_0,t_1], \quad  \int_{t_0}^{t_1}|\dot{x}(t)-\dot{x}_{*}(t)|dt\leq \varepsilon,\nonumber \\
&& |(u(t), \dot{x}(t))-(u_{*}(t), \dot{x}_*(t))|\leq R(t) \mbox{ a.e.} \,t\in [t_0,t_1].\label{radiusf}
\end{eqnarray}
This local minimum concept is even weaker than the so-called $W^{1,1}$ local minimum which is the case when $R(t)\equiv\infty$, because of the additional restriction (\ref{radiusf}) stemming  from the radius function. Note that $W^{1,1}$ local minimum is known to be weaker than the classical strong local minimum which has only the restriction that $| x(t)-x_{*}(t)|\leq\varepsilon \mbox{ a.e.}$.

 In \cite[Theorem 6.1]{cp}, Clarke and de Pinho obtained a set of necessary optimality conditions for problem $(P_{ICS})$ with $K_\varphi=\{0\}$ under the above concept of weak  local minimum. In  \cite[Theorem 2.1]{pinho*}, this result is extended to the problem $(P_{ICS})$ without the restriction of $K_\varphi=\{0\}$ under the classical strong local minimum concept. Moreover the result for the  smooth case  is further investigated in \cite{pinho*}. These necessary optimality conditions, however, require   the calibrated constraint qualification (CCQ)
 which is  stronger than the classical Mangasarian Fromovitz Condition (MFC) in optimal control theory, which is in turn stronger than the Mangasarian-Fromvitz constraint qualification (MFCQ) in mathematical programming. The main purpose of this paper is to derive necessary optimality conditions
 in the form of \cite[Theorem 6.1]{cp} and \cite[Theorem 2.1]{pinho*}  under weaker constraint qualifications.

 Following the same strategy   as proposed in     \cite{cp,pinho*}, by introducing a vector variable $v(t):=\dot{x}(t)$, we   transform $(P_{ICS})$ into the following equivalent problem:
\begin{eqnarray*}
(P_{ECS})~~~~~~\min &&  J(x,u):=\int_{t_0}^{t_1} F( x(t), u(t),v(t)) dt + f(x(t_0),x(t_1))\nonumber \\
s.t. && \dot{x}(t)=v(t)\,\quad a.e. \,t\in [t_0,t_1],\nonumber\\
&& \varphi(x(t), u(t),v(t))\in K_{\varphi}  \quad a.e. \,t\in [t_0,t_1], \nonumber\\
&& u(t) \in U \quad a.e. \,t\in [t_0,t_1],\nonumber\\
&& (x(t_0), x(t_1))\in S,\nonumber
\end{eqnarray*}
obtain a set of necessary optimality conditions for problem $(P_{ECS})$ and then transform back to the one for the original problem $(P_{ICS})$.
Problem $(P_{ECS})$ belongs to the class of optimal control { problems} with mixed state and control constraints. A set of necessary optimality conditions for a local minimum of radius $R(\cdot)$ for this class of problems  has been developed in  Clarke and de Pinho  \cite[Theorem 4.3]{cp} under the CCQ.
% Using the transformation $v=\dot{x}$, necessary optimality conditions for the optimal control of DAEs has been given in \cite[Theorem 6.1]{cp} which is extended to problem $(P_{ICS})$ by de Pinho in \cite[Theorem 2.1]{pinho*} under some stronger conditions.  The  result
% is  a set of necessary optimality conditions obtained under the calibrated constraint qualification (CCQ)
% %a geometric hypothesis called the {\em bounded slope condition} and
% which is  stronger than the classical Mangasarian Fromovitz Condition (MFC) in optimal control theory, which is in turn stronger than the Mangasarian-Fromvitz constraint qualification (MFCQ) in mathematical programming.
 %Although the issue of mixed constraints is broached in \cite{cp}, it is not completely developed. The %purpose of Section 3 is to do so.
%The set of necessary optimality condition is stratified in the sense that they are asserted on precisely the domain upon which the hypotheses (and the optimality) are assumed to hold.
Motivated by the recent progress in mathematical programming towards deriving necessary optimality conditions for mathematical programs under  constraint qualifications such as the calmness condition which is weaker than MFCQ, Li and Ye \cite{anli} proposed the so-called weak basic constraint qualification (WBCQ) plus the calmness of the perturbed constraint mapping
\begin{equation}
{M}_\varphi(\Theta):=\left\{(x,u,v)\in \mathbb{R}^{n_{x}}\times U \times \mathbb{R}^{n_{x}}:
%\left (\begin{array}{c}
                                           \varphi(x,u,v)+ \Theta\in K_{\varphi}
                                  %             \left (\begin{array}{c}K_{\varphi}
%                                           \\
%                                          \psi(x,u)\end{array}\right)+\Theta\in
%                                               \left (\begin{array}{c}K_{\varphi} \\
%                                               K_{\psi} \end{array}\right)
\right\},\label{perturb4.1n}
\end{equation} and obtained necessary optimality conditions for a local minimum of radius $R(\cdot)$ for the optimal control problem with mixed state and control constraints. Note that the concept of a local minimum of radius $R(\cdot)$ is slightly stronger than the one defined as in (\ref{radiusf}). In this paper we first show that result of \cite[Theorem 4.2]{anli} {remains true for the weaker local optimality concept in this paper} and apply it to $(P_{ECS})$ to obtain necessary optimality conditions of $(P_{ICS})$ under the desired constraint qualification.

In the case of DAEs with optimal controls lying in the interior of the control set, MFC is equivalent to the maximum rank of the Jacobian matrix {$\nabla_{\dot{x}} \varphi $}  and in the case of semi-explicit DAEs, it amounts to that the problem is index one.  Applying our results for the control of DAEs to the optimal control of   semi-explicit DAEs, we derive necessary optimality conditions for control of semi-explicit DAEs with index higher than one. In our necessary optimality conditions, the form of the  maximum principle for control of semi-explicit DAEs is  the weak maximum principle as in  \cite[Theorem 3.2]{pinhovinter} plus some extra condition called the Weierstrass condition. Hence in the autonomous case, our necessary optimality condition is a  maximum principle stronger  than  \cite[Theorem 3.2]{pinhovinter} under weaker constraint qualifications.

The paper is organized as follows. Section 2 contains preliminaries  on variational analysis.
In section 3, we derive  necessary optimality conditions for an autonomous  optimal control problems with mixed {state and control} constraints.
In section 4, we derive  necessary optimality conditions for  the optimal control of an implicit control system.
Optimal control of semi-explicit systems are studied in section 5.  In section 6 we give  verifiable sufficient conditions for the constraint qualifications required in the paper. The proof of the main result in section 3 is given in Appendix.

\section{Background in variational analysis}

In this section we present preliminaries  on  variational analysis that will be needed in this paper. We give only concise definitions
and conclusions that will be needed in the paper.  For more detailed information on the subject  we refer the reader to  \cite{c,clsw,anli,m2,rw}.

Throughout the paper, $|\cdot|$ denotes the Euclidean norm,  $B$ and $B(x,\delta)$  the open unit ball and the open ball centered at $x$ with radius $\delta>0$, respectively.
Unless otherwise specified,
the closure, the convex hull and the closure of the convex hull of a subset $\Omega\subseteq \mathbb{R}^n$ are denoted by $\bar \Omega$, co$\Omega$, and $\overline{\rm co}\Omega$, respectively. {For a set $\Omega \subseteq \mathbb{R}^n$ and a point $x\in \mathbb{R}^n$, $d(x,\Omega)$ is the distance from point $x$ to set $\Omega$.}  For any $a,b\in \mathbb{R}^n$, $\langle a, b \rangle $ denotes the inner product of vectors $a$ and $b$. Given a mapping $\psi: \mathbb{R}^n\rightarrow  \mathbb{R}^m$ and a point $x\in \mathbb{R}^n$, $\nabla \psi(x)\in R^{m\times n}$  stands for the Jacobian of $\psi(\cdot)$ at $x$. Given a function $f: \mathbb{R}^n\rightarrow  \mathbb{R} $, $\nabla^2 f(x)$ is the Hessian matrix. {For a set-valued map  $\Psi: \mathbb{R}^n\rightrightarrows \mathbb{R}^q$,  $gph \Psi:=\{(x,y): y\in \Psi(x)\}$ is its graph, $\Psi^{-1}(y):=\{x: y\in \Psi(x)\}$ is its inverse.}

%Given a nonempty closed subset $S\subseteq \mathbb{R}^n$ and a point $\bar{x}\in S$, we say that $\zeta\in \mathbb{R}^n$ is a {\em proximal normal vector} to $S$ at $\bar{x}$
%if there exists $\sigma=\sigma(\zeta,\bar{x})\geq 0$ such that $\langle\zeta,x-\bar{x}\rangle\leq \sigma|x-\bar{x}|^2, \forall x\in S$,
%where $\langle a, b\rangle$ denotes the inner product of vectors $a$ and $b$, $|\cdot|$ denotes the Euclidean norm. The set of such $\zeta$,
%denoted by $N^{P}_{S}(\bar{x})$, is termed {\em the proximal normal cone} to $S$.
%The {\em limiting normal cone} $N^{L}_{S}(\bar{x})$ to $S$ is defined by
%$$N^{L}_{S}(\bar{x}):=\{\lim\zeta_{i}:\zeta_{i}\in N^{P}_{S}(x_{i}),x_{i} \xrightarrow{ S} {\bar{x}}\},$$
%where $x_{i} \xrightarrow{ S} {\bar{x}}$ means that $x_i\in S$ and $x_i\rightarrow \bar{x}$.
%The (Bouligand) tangent cone to $S$ at $\bar{x}$ is defined by
%$$T_{S}(\bar{x}):=\{w\in \mathbb{R}^{n}: \exists t_{k}\downarrow 0, w_{k}\rightarrow w \mbox { with }  \bar{x}+t_{k}w_{k}\in S, \forall k \}.$$
Let $S\subseteq \mathbb{R}^n$. The {\em tangent cone} to $S$ at $\bar{x}$ is defined by
$$T_{S}(\bar{x}):=\{w\in \mathbb{R}^{n}: \exists t_{k}\downarrow 0, w_{k}\rightarrow w \mbox { with }  \bar{x}+t_{k}w_{k}\in S, \forall k \}.$$ The {\em Fr\'{e}chet normal cone} to $S$ at $\bar{x}\in S$ is defined by
$$\hat{N}_{S}(\bar{x}):=\{v^{*}\in \mathbb{R}^{n}:\limsup\limits_{x \xrightarrow{ S} \bar{x}}\frac{\langle v^{*}, x-\bar{x}\rangle}{|x-\bar{x}|}\leq 0\},$$
where $x_{i} \xrightarrow{ S} {\bar{x}}$ means that $x_i\in S$ and $x_i\rightarrow \bar{x}$.
The {\em limiting normal cone} $N_{S}(\bar{x})$ to $S$ is defined by
$$N_{S}(\bar{x}):=\{\lim\zeta_{i}:\zeta_{i}\in \hat{N}_{S}(x_{i}),x_{i} \xrightarrow{ S} {\bar{x}}\}.$$
{$S$ is said to be  normally regular if $\hat{N}_{S}(\bar{x})=N_{S}(\bar{x})$ for all $\bar{x}\in S$.}
Recently Gfrerer \cite{helmut13} introduced the concept of the directional limiting normal cone. The {\em limiting  normal cone to $S$ in direction $w\in \mathbb{R} ^{n}$} at $\bar{x}$ is defined by
$$N_{S}(\bar{x};w):=\{v^{*}\in \mathbb{R}^{n}: \exists t_{k}\downarrow 0, w_{k}\rightarrow w,  v_{k}^{*}\rightarrow  v^{*} \mbox { s.t. }v_{k}^{*}\in \hat{N}_{S}(\bar{x}+t_{k}w_{k}), \forall k\}.$$

Consider a lower semicontinuous function $f:\mathbb{R}^n\rightarrow \mathbb{R}\cup\{+\infty\}$ and a point $\bar{x}\in \mathbb{R}^n$ where $f$ is finite.
A vector $\zeta\in \mathbb{R}^n$ is called a {\em proximal subgradient} of $f$ at $\bar{x}$ provided that there exist $\sigma,\delta>0$ such that
$$f(x)\geq f(\bar{x})+\langle\zeta,x-\bar{x}\rangle-\sigma|x-\bar{x}|^{2},\forall x\in B(\bar{x},\delta).$$
The set of such $\zeta$  is denoted $\partial^Pf(\bar{x})$ and referred to as the {\em proximal subdifferential}.
The {\em limiting subdifferential} of $f$ at $\bar{x}$  is the set
$$\partial f(\bar{x}):=\{\lim\zeta_{i}:\zeta_{i}\in \partial^Pf(x_{i}),x_{i}\rightarrow \bar{x}, f(x_{i})\rightarrow f(\bar{x})\}.$$
For a locally Lipschitz function $f$ on $\mathbb{R}^n$,  the {\em generalized gradient} $\partial^Cf(\bar{x})$ coincides with
$co\partial f(\bar{x})$; further the associated {\em Clarke normal cone}
$N^{C}_{S}(\bar{x})$ at $\bar{x}\in S$ coincides with $\overline{co}N_{S}(\bar{x})$.

%Let  $\Psi: \mathbb{R}^n\rightrightarrows\Re^q$ be an arbitrary set-valued map (assigning to each $z\in \mathbb{R}^n$, a set
%$\Psi(z)\subseteq \mathbb{R}^q$ which may be  empty).  The graph  of $\Psi$ is denoted by  $gph\Psi$ (i.e., $gph \Psi:=\{(z,v): v\in \Psi(z)\}$). Let  $(\bar{z},\bar{v})\in \overline{gph \Psi}$. The set-valued map $D^*\Psi(\bar{z},\bar{v})$
%from $\mathbb{R}^q$ into $\mathbb{R}^n$ defined by
%$$D^*\Psi(\bar{z},\bar{v})(\eta)=\{\xi\in \mathbb{R}^n:(\xi,-\eta)\in N_{gph\Psi}^{L}(\bar{z},\bar{v})\}$$
%is called the coderivative of $\Psi$ at the point $(\bar{z},\bar{v})$. The symbol $D^*\Psi(\bar{z})$ is used when $\Psi$ is single-valued
%at $\bar{z}$ and $\bar{v}=\Psi(\bar{z})$. In particular if $\Psi: \mathbb{R}^n\rightarrow \mathbb{R}^q$ is single-valued and Lipschitz near $\bar{z}$, then
%$$D^*\Psi(\bar{z})(\eta)=\partial^L\langle\eta, \Psi \rangle (\bar{z})\quad \forall \eta \in \mathbb{R}^q.$$

We now review some concepts of Lipschitz continuity of set-valued maps.
\begin{defn}\cite{robin}
A set-valued map $\Psi: \mathbb{R}^n\rightrightarrows \mathbb{R}^q$ is said to be upper-Lipschitz at $\bar{x}$ if
there exist $\mu\geq 0$ and a neighborhood $U(\bar{x})$ of $\bar{x}$ such that
$$\Psi(x)\subseteq \Psi(\bar{x})+\mu|x-\bar{x}|\bar{B},  \,\,\forall x\in U(\bar{x}).$$
\end{defn}
\begin{defn}\cite[Definition 1.40]{m2} \label{def2.2}
A set-valued map $\Psi: \mathbb{R}^n\rightrightarrows \mathbb{R}^q$ is said to be pseudo-Lipschitz (or locally Lipschitz like or has the Aubin property) around $(\bar{x},\bar{y})\in gph\Psi$ if there exist
 $\mu\geq 0$ and  neighborhoods $U(\bar{x})$, $U(\bar{y})$ of $\bar{x}$ and $\bar{y}$, respectively, such that
$$\Psi(x)\cap U(\bar{y})\subseteq \Psi({x'})+\mu|x-{x'}|\bar{B},  \,\,\forall x,x'\in U(\bar{x}).$$
{Equivalently, $\Psi$ is pseudo-Lipschitz around $(\bar{x},\bar{y})$ if there exist
 $\mu\geq 0$ and neighborhoods $U(\bar{x})$, $U(\bar{y})$ of $\bar{x}$ and $\bar{y}$, respectively, such that
 $$d(y, \Psi({x'}) )\leq \mu d({x'}, \Psi^{-1}(y)) \qquad \forall  x'\in U(\bar{x}), y\in  U(\bar{y}).$$}
%The number $\mu$ satisfying the above inclusion is called a modulus.
% and the infimum of all such moduli $\{\mu\}$ is called the exact Lipschitz bound of $\Psi$ around $(\bar{z},\bar{v})$ and is denoted by ${\rm lip}\Psi(\bar{z},\bar{v})$.
%More explicitly,  $\Psi$  is said to be satisfy a pseudo-Lipschitz condition of radius $R$ near $\bar{z}$ if there exists $\varepsilon>0$ and  $\mu>0$ such that,
%$$z,z'\in B(\bar{z},\varepsilon) \rightrightarrows \Psi(z)\cap\bar{ B}(\bar{z},R)\subseteq  \Psi({z'})+k(t)|z-z'|\bar{B}.$$
\end{defn}
\begin{defn} \cite{y05,rw}
A set-valued map $\Psi: \mathbb{R}^n\rightrightarrows \mathbb{R}^q$ is said to be calm (or pseudo upper-Lipschitz continuous) at $(\bar{x},\bar{y})\in gph\Psi$ if
there exist $\mu\geq 0$ and neighborhoods $U(\bar{x})$, $U(\bar{y})$ of $\bar{x}$ and $\bar{y}$, respectively, such that
$$\Psi(x)\cap U(\bar{y})\subseteq \Psi(\bar{x})+\mu|x-\bar{x}|\bar{B},  \,\,\forall x\in U(\bar{x}).$$
{Equivalently, $\Psi$ is calm  around $(\bar{x},\bar{y})$ if there exist
 $\mu\geq 0$ and a neighborhood $U(\bar{y})$ of  $\bar{y}$  such that
 $$d(y, \Psi({\bar{x}}) )\leq \mu d(\bar{x}, \Psi^{-1}(y)) \qquad \forall  y\in  U(\bar{y}).$$}
\end{defn}
%Although the term ``calmness'' was coined in Rockafellar and Wets \cite{rw}, the concept of
%calmness of a set-valued map was first introduced by Ye and Ye in \cite{y05} under the term
%``pseudo upper-Lipschitz continuity'' which comes from the fact that it is a combination of
%Aubin's pseudo-Lipschitz continuity \cite{aubin} and Robinson's upper-Lipschitz continuity \cite{robin}.
%It is easy to see that   the calmness is implied by either upper-lipschitz
%continuity or pseudo-Lipschitz continuity  (see \cite{y05}).
\begin{defn} \cite{Ioffe}
{A set-valued map $\Sigma: \mathbb{R}^q \rightrightarrows \mathbb{R}^n$ is said to be metrically subregular at $(\bar{y},\bar{x})\in  gph\Sigma$  if there exist
 $\mu\geq 0$ and a neighborhood $U(\bar{y})$ of  $\bar{y}$  such that
 $$d(y, \Sigma^{-1}({\bar{x}}) )\leq \mu d(\bar{x}, \Sigma(y)) \qquad \forall  y\in  U(\bar{y}).$$}
\end{defn}
{From definition, it is easy to see that a set-valued map $\Sigma$ is metrically subregular at $(\bar{y},\bar{x})\in  gph\Sigma$  if and only if its inverse map $\Sigma^{-1}$ is calm at $(\bar{x},\bar{y})\in  gph\Sigma^{-1}$.}

{In this paper we are mostly interested in the calmness of a set-valued map
defined as the perturbed constrained system:
\begin{equation}
M(\Theta):=\{(x,u)\in \mathbb{R}^{n_{x}}\times U: \Phi(x,u) +\Theta \in \Omega\},\label{perturb3.1}
\end{equation}
where $\Phi:\mathbb{R}^{n_{x}}\times \mathbb{R}^{n_{u}}\rightarrow  \mathbb{R}^d$ and $U\subseteq \mathbb{R}^{n_{u}}, \Omega \subseteq \mathbb{R}^d.$

We now summarize some constraint qualifications that will be used in the paper.
\begin{defn}\label{defn2.5}  Let $(\bar x,\bar{u})\in M(0)$, $\Phi$ is Lipschitz continuous at $(\bar x,\bar u)$ and $U, \Omega$ are closed.
\begin{itemize}
%\item
%We say that the linear CQ holds if $\Phi$ is affine and $U, \Omega$ are the union of finitely many polyhedral sets.

\item (\cite{cp}) We say  the calibrated constraint qualification (CCQ)  holds at $(\bar x, \bar u)$ if there exists $\mu >0$ such that
\begin{eqnarray*}
\left \{ \begin{array}{l}
(\alpha,\beta)\in \partial\langle \lambda, \Phi\rangle(\bar{x},\bar{u})+\{0\}\times N_{U}(\bar{u}),\\
\lambda\in N_{\Omega}(\Phi(\bar{x},\bar{u}))
\end{array} \right.  && \Longrightarrow |\lambda|\leq \mu |\beta|.
\end{eqnarray*}
\item (\cite{cp} We say  the MFC holds at $(\bar x, \bar u)$ if
\begin{eqnarray*}
\left \{ \begin{array}{l}
(\alpha,0)\in \partial\langle \lambda, \Phi\rangle(\bar{x},\bar{u})+\{0\}\times N_{U}(\bar{u}),\\
\lambda\in N_{\Omega}(\Phi(\bar{x},\bar{u}))
\end{array} \right.  && \Longrightarrow\lambda=0.
\end{eqnarray*}
\item (\cite{m2}) We say  the no nonzero abnormal multiplier constraint qualification (NNAMCQ) holds at $(\bar x, \bar u)$ if
\begin{eqnarray*}
\left \{ \begin{array}{l}
(0,0)\in \partial\langle \lambda, \Phi\rangle(\bar{x},\bar{u})+\{0\}\times N_{U}(\bar{u}),\\
\lambda\in N_{\Omega}(\Phi(\bar{x},\bar{u}))
\end{array} \right.  && \Longrightarrow\lambda=0.
\end{eqnarray*}
\item (\cite{anli}) We say  the weak basic constraint qualification (WBCQ) holds at $(\bar x, \bar u)$ if
\begin{eqnarray*}
\left\{ \begin{array}{l} (\alpha,0)\in \partial\langle \lambda, \Phi\rangle(\bar{x},\bar{u})+\{0\}\times N_{U}(\bar{u}),\\
\lambda\in N_{\Omega}(\Phi(\bar{x},\bar{u}) )
\end{array} \right. \qquad \Longrightarrow\alpha=0. \label{cq1section2}
\end{eqnarray*}
\end{itemize}
\end{defn}
It is easy to check that the following implications hold:
$$ \mbox{CCQ} \Longrightarrow  \mbox{MFC} \Longleftrightarrow \mbox{WBCQ+NNAMCQ}  \Longrightarrow \mbox{WBCQ + Calmness of $M$},$$
and the WBCQ$+$Calmness of $M$ may not imply NNAMCQ (see \cite[Example 2.1]{anli}).}
Although in general CCQ is stronger than MFC, if MFC holds for every point in certain compact set, then it implies CCQ
for every point in the same compact set under certain assumptions; see \cite[Proposition 4.6]{cp} for details.

%Next we briefly discuss  sufficient conditions for the calmness of $M$ defined as in (\ref{perturb3.1}). It is well-known that the linear CQ is a sufficient condition for the calmness of the set-valued map $M$, and NNAMCQ is a sufficient condition for the pseudo Lipschitz continuity of the set-valued map $M$.
%Although the linear CQ and NNAMCQ are easy to verify, they may be still too strong for some problems to hold.
%Recently some new constraint qualifications that are stronger than calmness and
%weaker than the linear CQ and/or NNAMCQ for standard nonlinear
%programs with equality and inequality constraints have been introduced in the literature.
%These CQs include the relaxed
%constant positive linear dependence condition (see \cite{abdreabi-haeser--}), the
%constant rank of the subspace component condition(see \cite[Theorem 4.2]{abdeabi--twoCQ}) and
%the quasinormality \cite[Theorem 5.2]{guoyezhang}.
%
%Note that the calmness of the set-valued map $M $ is equivalent to the metric subregulairy of the set-valued map $(\Phi(x,u), u )\in \Omega\times U$. Recently by using the directional normal cones, some new sufficient conditions for metric subregularity  have been developed (see \cite{helmut,helmutYe,adi}).
%

\section{Optimal control problems with mixed state and control constraints}
\label{section3}

In this section,  we consider   the following autonomous optimal control problem  in which the
 state and control variables are subject to mixed {state and control} constraints:
\begin{eqnarray*}
(P)~~~~~~\min &&  J(x,u):=\int_{t_0}^{t_1} F( x(t), u(t)) dt + f(x(t_0),x(t_1))\nonumber \\
s.t. && \dot{x}(t)=\phi(x(t), u(t)) \,\quad a.e. \,t\in [t_0,t_1],\nonumber\\
&& \Phi(x(t), u(t))\in \Omega \quad a.e. \,t\in [t_0,t_1], \\
&& u(t) \in U \quad a.e. \,t\in [t_0,t_1],\nonumber\\
&& (x(t_0), x(t_1))\in S,\nonumber
\end{eqnarray*}
{where $F:\mathbb{R}^{n_{x}}\times \mathbb{R}^{n_{u}}\rightarrow \mathbb{R}$, {$f:\mathbb{R}^{n_{x}}\times \mathbb{R}^{n_{x}}\rightarrow \mathbb{R}$,}
$\phi:\mathbb{R}^{n_{x}}\times \mathbb{R}^{n_{u}}\rightarrow \mathbb{R}^{n_{x}}$, $\Phi:\mathbb{R}^{n_{x}}\times \mathbb{R}^{n_{u}}\rightarrow  \mathbb{R}^d$ and $U\subseteq \mathbb{R}^{n_u}, \Omega \subseteq \mathbb{R}^d, S \subseteq \mathbb{R}^{n_x}\times  \mathbb{R}^{n_x}$. Unless otherwise stated, in this section we assume that  $F, f,
\phi, \Phi$ are locally Lipschitz continuous, and the sets
 $U,\Omega$, $S$ are  closed.}

Let $R:[t_0,t_1]\rightarrow (0,+\infty]$ be a given measurable radius function.
As in \cite{cp}, we say that  an admissible pair $(x_{*},u_{*})$ is a  local minimum of radius $R(\cdot)$ for problem $(P)$ if it  minimizes the value of the cost
function $J(x, u)$ over all admissible pairs $(x,u)$ which satisfies
$$| x(t)-x_{*}(t)|\leq\varepsilon, \,|u(t)-u_{*}(t)|\leq R(t) \mbox{ a.e.,} \int_{t_0}^{t_1}|\dot{x}(t)-\dot{x}_{*}(t)|dt\leq \varepsilon.$$
For any given $\varepsilon>0$ and a given radius function $R(\cdot)$, define
\begin{eqnarray*}
&&\tilde{S}_*^{\varepsilon,R}(t):=\{(x,u)\in  \bar{B}(x_*(t),\varepsilon)\times U: \Phi(x,u)\in \Omega,   |u-u_*(t)|\leq R(t)\},\\
&&\tilde{C}_*^{\varepsilon,R}:=cl\{(t,x,u)\in [t_0,t_1]\times \mathbb{R}^{n_{x}}\times \mathbb{R}^{n_{u}}:  (x,u)\in \tilde{S}_*^{\varepsilon,R}(t)\},
\end{eqnarray*}
where $cl$ denotes the closure. {In the case where the control set $U$ is  {closed},  the optimal control $u_*(t)$  is continuous and the radius function $R(t)$ is  either identical to $\infty$ or continuous, the closure operation is superfluous and hence can be removed. A sufficient condition for the compactness of the  set $\tilde{C}_*^{\varepsilon,R}$ is that {$\varepsilon
<\infty$ and either $U$ is compact or $u_*(t)$ is continuous and  $R(t)$ is either identical to $\infty$ or continuous.}

The main result of this section is the following theorem {whose proof can be found in the appendix.}
\begin{thm}\label{thm4.2new}  Let  $(x_*,u_*)$ be a local minimum of radius $R(\cdot)$ for $(P)$.
 Suppose that there exists $\delta>0$ such that $R(t)\geq \delta$.
  Suppose that $\tilde{C}_*^{\varepsilon,R}$ is compact and  for
 all $(t,x,u) \in \tilde{C}_*^{\varepsilon, R}$  the WBCQ holds:
\begin{eqnarray}
\left \{ \begin{array}{l}  (\alpha,0)\in \partial\langle \lambda, \Phi\rangle(x,u)+\{0\}\times N_{U}(u),\\
\lambda\in N_{\Omega}(\Phi(x,u))\end{array} \right. \Longrightarrow \alpha=0 \label{WBCQ1}
\end{eqnarray}
and the mapping $M$ defined as in (\ref{perturb3.1}) is calm at $(0, x,u)$.
Then there exist an arc $p$ and   a number $\lambda _{0}$ in $\{0,1\}$,
%and  a measurable function $\lambda:[t_0,t_1]\rightarrow \mathbb{R}^n$ with
% for almost every $t\in [t_0,t_1]$
satisfying the nontriviality condition
$(\lambda _{0},p(t))\neq0, \forall t\in[t_0,t_1]$,
 the transversality condition
$$(p(t_0),-p(t_1)) \in \lambda_0 \partial f(x_*(t_0),x_*(t_1))+N_S (x_*(t_0),x_*(t_1)),$$
and the Euler adjoint inclusion for almost every $t$:
\begin{eqnarray}
\lefteqn{(\dot{p}(t), 0) \in \partial^C \{\langle -p(t), \phi\rangle+\lambda_0F\} ( x_*(t),u_*(t))+\{0\}\times N^C_{U}(u_*(t))}\nonumber \\
&& +co\{ \partial\langle  \lambda,  \Phi\rangle ( x_*(t),u_*(t)): \lambda\in N_{\Omega} (\Phi(x_*(t),u_*(t)){)}\}, \label{EulerNnew}
\end{eqnarray}
as well as the Weierstrass condition of radius $R(\cdot)$ for almost every $t$:
\begin{eqnarray*}
&&\Phi(x_*(t),u)\in \Omega, u\in U, \, \,  |u-u_*(t)|< R(t)\Longrightarrow\\
&&\langle p(t),\phi(x_*(t),u) \rangle-\lambda_{0}F(x_*(t),u)\leq \langle p(t), \phi (x_*(t),u_*(t)) \rangle -\lambda_{0}F(x_*(t),u_*(t)).
\end{eqnarray*}
Moreover in the case of free end point, $\lambda_0$ can be taken as $1$.
\end{thm}
{For the autonomous control problem $(P)$, the conclusions of Theorem \ref{thm4.2new} are  exactly the same as those in Clarke and {de} Pinho \cite[Theorem 4.3]{cp} except that the Weierstrass condition holds only on the open ball $B(u_*(t),R(t))$ instead of the closed ball $\bar B(u_*(t), R(t))$. However our assumption that the WBCQ plus the calmness condition is weaker than the calibrated constraint qualification in \cite[Theorem 4.3]{cp}, which is even stronger than the MFC. In fact, the Weierstrass conditions in \cite{cr,cp}  can  only hold on the open ball $B(u_*(t),R(t))$ instead of the closed ball $\bar B(u_*(t), R(t))$. This imprecision was spotted and remedied
 in \cite{b}. Moreover the authors in \cite{b} introduced a notion of {\it radius multifunction} and used it to consider a more general  concept of a local minimum and necessary optimality conditions.}

The Euler adjoint inclusion (\ref{EulerNnew}) in Theorem \ref{thm4.2new} is in an implicit form. In the case where $\Phi$ is smooth,
one can find a measurable  multiplier $\lambda(t)\in N_{\Omega}^{C}(\Phi(x_*(t),u_*(t))$ such that the Euler adjoint inclusion takes an explicit multiplier form by using the measurable selection theorem.

To give an estimate for the multiplier $\lambda$ we need to use the following result.
\begin{proposition}\cite[Proposition 4.1]{helmut3} \label{Prop3.2}Let $\Psi: \mathbb{R}^n \rightrightarrows \mathbb{R}^q$ be a set-valued map with closed graph. Given $(\bar x,\bar y)\in gph \Psi$, assume that $\Psi$ is  metrically subregular at $(\bar x,\bar y)$ with modulus $\kappa$. Then
$$N_{\Psi^{-1}(\bar y)}( \bar x)\subseteq \{ \gamma: \exists \lambda \in \kappa |\gamma| \bar{B}: (\gamma,\lambda)\in N_{gph \Psi} (\bar x,\bar y)\}.$$
\end{proposition}
We are now in a position to give the Euler adjoint inclusion an explicit multiplier form when $\Phi$ is smooth. Moreover in the case where $\Phi$ smooth and
{ $u_{*}(t)$ is in the interior of $U$ for almost all $t$,} we show that a multiplier can be chosen such that an estimate in terms of adjoint arc holds as in  \cite[Theorems 4.3]{cp}. Our result improves the corresponding result in \cite[Theorems 4.3]{cp} in that for the autonomous case, the estimate holds under the WBCQ plus the calmness condition which is weaker than the calibrated constraint qualification required in \cite[Theorems 4.3]{cp}.
%First we recall the following result.
\begin{thm}\label{New}  In additions to  the assumptions of Theorem \ref{thm4.2new},
 suppose that $\Phi$ is strictly differentiable. Then the Euler adjoint inclusion can be replaced by the one in the explicit multiplier form, i.e., there exists
 a { measurable}  function $\lambda:[t_0,t_1]\rightarrow \mathbb{R}^d$ with
$\lambda(t)\in N_{\Omega}^{C}(\Phi(x_*(t),u_*(t)))$ for almost every $t\in [t_0,t_1]$
satisfying
\begin{eqnarray}\label{222}
\lefteqn{(\dot{p}(t),0)\in \partial^C \{\langle -p(t), \phi\rangle+\lambda_0F\}
 ( x_*(t),u_*(t)) }\nonumber \\
&&
+ \nabla\Phi(x_*(t),u_*(t))^T\lambda(t) + \{0\}\times N^C_{U}(u_*(t)).
\end{eqnarray}
Moreover {if $N_U^C(u_*(t))=\{0\}$
  and $\Omega$ is normally regular,
 then  the multiplier $\lambda(t)$ }can be chosen such that  the following estimate holds:
\begin{equation}\label{3.9}
|{\lambda}(t)|\leq \kappa\{(k+k^{\phi})|p(t)|+\lambda_0k^{F}\}\qquad a.e.
\end{equation}
for some positive constants $k, \kappa, k^{\phi}, k^{F} $, where  $k^{\phi}$, $k^{F}$ are the Lipschitz coefficients of $\phi,F$ on set $D$ defined as in (\ref{setD}) respectively.
\end{thm}
\begin{proof} By \cite[Theorem 14.26]{rw}, one can easily get the measurability of the mapping $\lambda: t\rightarrow N_{\Omega}^{C}(\Phi(x_*(t),u_*(t)))$.
The Euler adjoint inclusion in the explicit multiplier form can be  easily verified in  (\ref{EulerNnew}) when $\Phi$ is  strictly differentiable.
%and $U=\mathbb{R}^{n_{u}}$.
%Next we prove $\lambda(t)$ is bounded.
%We suppose that $t$ is a point where the other results of Theorem \ref{New} hold.

We now prove the estimate for $\lambda(t)$ in (\ref{3.9}). Since the set-valued map $M$ is calm at $(0,x_*(t),u_*(t))$, it is equivalent to saying that the set-valued map $M^{-1}(x,u):=\Phi(x,u)-\Omega$ is metrically subregular at $(x_*(t),u_*(t),0)$.
% Then there exists an open neighborhood W of $(x_*(t),u_*(t))$ and a real number $\kappa>0$ such that
%$$d((x,u), M(0))\leq \kappa d((x_*(t),u_*(t)),M^{-1}(x,u)), \forall (x,u)\in W.$$
Since the set
\begin{equation}\label{setD}
D:=cl\left \{\cup_{t\in[t_0,t_1]} (x_*(t),u_*(t)) \right\}    
\end{equation}  is compact, one can find a constant $\kappa>0$  such that
 the set-valued map $M^{-1}:=\Phi(x,u)-\Omega$ is  metrically subregular at $(x_*(t),u_*(t) ,0)$ for all $(x_*(t),u_*(t)\textcolor{red}{)} \in D$ with the same modulus $\kappa>0$.
We get  by Proposition \ref{Prop3.2} that
\begin{eqnarray*}
\lefteqn{ N_{M(0)}(x_*(t),u_*(t)) }\\
&& \subseteq \{(\alpha,\beta): \exists  -\lambda \in \kappa |(\alpha,\beta)|\bar B\mbox{ s.t. } (\alpha,\beta, -\lambda)\in  N_{gphM^{-1}}((x_*(t),u_*(t),0))\}.
%partial^L\hspace{-0.7cm}&&\langle \lambda, \Phi(x_*(t),u_*(t))\rangle:\\&&
%\lambda\in N_{\Omega}^{L}(\Phi(x_*(t),u_*(t)))\}+N_{\mathbb{R}^n\times U}^{L}(x_*(t),u_*(t)).
\end{eqnarray*}
Since $gphM^{-1}=\{(x,u,\nu): \nu\in \Phi(x,u)-\Omega\}=\{(x,u,\nu): \Phi(x,u)-\nu\in \Omega\}$,  it follows from \cite[Exercise 6.7]{rw} that
\begin{eqnarray*}
\lefteqn{N_{gphM^{-1}}(x_*(t),u_*(t),0)=}\\
&& \{(\alpha,\beta,-\lambda): (\alpha,\beta)=\nabla\Phi(x_*(t),u_*(t))^{T}\lambda, \lambda\in N_{\Omega}(\Phi(x_*(t),u_*(t))) \}.\end{eqnarray*}
%Letting $\tilde{\lambda}:=-\lambda$, then we obtain that
%Since $\Omega$ is normally regular,   $N_{\Omega}(\Phi(x_*(t),u_*(t)))=N_{\Omega}^{C}(\Phi(x_*(t),u_*(t)))$ is convex.
Therefore
\begin{eqnarray}\label{333}
\lefteqn{ N_{M(0)}(x_*(t),u_*(t)) }\nonumber\\
&&\subseteq\{(\alpha,\beta):  \exists \tilde{\lambda} \in \kappa |(\alpha,\beta)|\bar B\cap N_{\Omega}(\Phi(x_*(t),u_*(t))),(\alpha,\beta)=\nabla\Phi(x_*(t),u_*(t))^{T} \tilde{\lambda}\}.\nonumber \\
&& \qquad \qquad \qquad
\end{eqnarray}

%From (\ref{222}), we can choose  $(\zeta(t),\eta(t))\in \partial^C\{\langle -p(t),\phi(\cdot,\cdot)\rangle\}(x_*(t),u_*(t))$ satisfying $(\dot{p}(t),0)-(\zeta(t),\eta(t))= \nabla\Phi(x_*(t),u_*(t))^T\lambda(t)$.
%In view of the Lipschitz assumption on $\phi$, we get that  $|(\zeta(t), \eta(t))|\leq k^{\phi}|p(t)|$, where $k^{\phi}$ are the Lipschitz coefficient of $\phi$.
%From the proof of Theorem \ref{thm4.2new}, we obtain that $|\dot{p}(t)|\leq k |p(t)|$, where $k>0$ is a Pseudo-Lipschitz modulus of the set-valued mapping $\Gamma$.
%Then we get that by (\ref{333})
%$$\lambda(t)\leq \kappa(|(\dot{p}(t),0)|+|(\zeta(t),\eta(t))|)\leq \kappa( k+k^{\phi})|p(t)|.$$

Since  the proof of Theorem \ref{thm4.2new} is based on  Proposition \ref{Propnew} which is \cite[Theorem 4.2]{anli} whose proof is based on transforming the optimal control problem to a differential inclusion problem with a pseudo-Lipschitz set-valued {map}, we can obtain that $|\dot{p}(t)|\leq k |p(t)|$  where  constant $k>0$ is the pseudo-Lipschitz module of the set-valued map.
%(the proof is similar as in \cite[Remark 3.1]{anli}).
%where $k>0$ is a Pseudo-Lipschitz modulus of the set-valued mapping $\Gamma$.
Moreover since $\Omega$ is normally regular, the limiting normal cone coincides with the Clarke  normal cone to $\Omega$. Hence from the proof of \cite[Theorems 4.1 and 4.2]{anli}, if we use the estimate in  (\ref{333}) to replace the estimate for $N_{M(0)}(x_*(t),u_*(t))$, then   for almost every  $t$, we can find $\tilde{\lambda}(t) \in \kappa |\nabla\Phi(x_*(t),u_*(t))^{T}\tilde{\lambda}(t)|\bar B\cap
N_{\Omega}(x_*(t),u_*(t))$  satisfying the Euler's inclusion:
$$(\dot{p}(t),0)\in \partial^C \{\langle -p(t), \phi\rangle+\lambda_0F\}
 ( x_*(t),u_*(t))
+ \nabla\Phi(x_*(t),u_*(t))^T\tilde\lambda(t).$$ From this Euler's inclusion,  we may choose   $$(\zeta(t),\eta(t))\in \partial^C\{\langle -p(t),\phi\rangle+\lambda_0F\}(x_*(t),u_*(t))$$ satisfying $(\dot{p}(t),0)-(\zeta(t),\eta(t))= \nabla\Phi(x_*(t),u_*(t))^T \tilde \lambda(t)$.
In view of the Lipschitz assumption on $\phi,F$ and the compactness of set $D$, we get that  $|(\zeta(t),\eta(t))|\leq k^{\phi}|p(t)|+\lambda_0k^{F}$, where $k^{\phi}$, $k^{F}$ are the Lipschitz coefficients of $\phi,F$ with respected to $(x,u)$ on set $D$ respectively. It follows that
\begin{eqnarray*}
|\tilde\lambda(t)|
& \leq & \kappa  |\nabla \Phi(x_*(t),u_*(t))^T \tilde\lambda(t)|\\
& \leq & \kappa |(\dot{p}(t),0)-(\zeta(t),\eta(t))|  \\
&\leq & \kappa k |p(t)|+
\kappa(k^{\phi}|p(t)|+\lambda_0k^{F})\leq \kappa\{(k+k^{\phi})|p(t)|+\lambda_0k^{F}\}, \, \mbox{ a.e. }
\end{eqnarray*}
\end{proof}

The constraint qualification imposed in Theorem \ref{thm4.2new} is required to hold for points in a neighborhood of the optimal process $(x_*,u_*)$.
It is natural to ask whether this condition can be imposed only along the optimal process $(x_*,u_*)$.
In order to answer this question we first introduce the following concept.
\begin{defn} \cite[Definition 4.7]{cp}
We say that $(t,x_*(t),u)$ is an admissible cluster point of $(x_*,u_*)$ if there exists a sequence $t_i\in [t_0,t_1]$ converging to $t$ and  $\Phi(x_i,u_i)\in  \Omega$, $u_i\in U$ such that $\lim x_i=x_*(t)$ and $\lim u_i=\lim u_*(t_i)=u$.
\end{defn}
We now derive a similar result as Clarke and  {de} Pinho \cite[Theorem 4.8]{cp}  under  the WBCQ plus the calmness of $M$ which is weaker than MFC required by \cite[Theorem 4.8]{cp}. Note that in the case where $u_*(t)$ is continuous, the only admissible cluster point of $(x_*,u_*)$ is $(t, x_*(t),u_*(t))$ and hence the constraint qualification is only needed to be verified along the optimal process $(x_*,u_*)$.

\begin{thm}\label{thm4.3} Let  $(x_*,u_*)$ be a  local minimum of constant radius $R$ for $(P)$.
Suppose that  the optimal control $u_{*}$ is bounded.
Assume that for every $(x_*(t),u)$ such that $(t,x_*(t),u)$ is an admissible cluster point of $(x_*,u_*)$,  the WBCQ holds:
\begin{eqnarray*}
\left \{ \begin{array}{l}  (\alpha,0)\in \partial\langle \lambda, \Phi\rangle(x_*(t),u)+\{0\}\times N_{U}(u),\\
\lambda\in N_{\Omega}(\Phi(x_*(t),u))\end{array} \right. \Longrightarrow \alpha=0
\end{eqnarray*}
and the map $M$ defined as in (\ref{perturb3.1}) is calm at $(0, x_*(t),u)$.
%Suppose also that the tempered growth condition for radius $R$ near $x_*$ holds.
%If for some positive $r,c$, we have, for almost every $t$ and $(x,u)\in S_*^{\varepsilon,R}(t)$,
%$$ \bar{B}(0;r)\subseteq \Omega(t)-\{H(t,x,u)-\langle D_{u}H(t,x,u)+\zeta,u'-u\rangle:\zeta\in N_{U(t)}^L(u), |u'-u|\leq c\},$$
Then the necessary optimality conditions of Theorem \ref{thm4.2new} hold as stated with some radius $\eta \in (0,R)$: for some $\eta  \in (0,R)$, for $t$ $ a.e.$,
\begin{eqnarray*}
&&\Phi(x_*(t),u)\in \Omega, u\in U, \, \,  |u-u_*(t)|< \eta\Longrightarrow\\
&&\langle p(t),\phi(x_*(t),u) \rangle-\lambda_{0}F(x_*(t),u)\leq \langle p(t), \phi (x_*(t),u_*(t)) \rangle -\lambda_{0}F(x_*(t),u_*(t)).
\end{eqnarray*}
Moreover if $u_*(\cdot)$ is continuous, then the WBCQ and the calmness condition are only required to hold along $(x_*(t),u_*(t)).$

Moreover if  $\Phi$ is strictly differentiable,
then the Euler adjoint inclusion can be replaced by the one in the explicit multiplier form {(\ref{222})} and if $N_U^C(u_*(t))=\{0\}$  and $\Omega$ is normally regular, then  the estimate for the  multiplier $\lambda(t)$ in {(\ref{3.9})} also holds.
\end{thm}
The proof of Theorem \ref{thm4.3} uses the following result.
\begin{proposition}\cite[Theorem 4.3]{anli}\label{Propnew1}
 Let  $(x_*,u_*)$ be a $W^{1,1}$ local minimum of constant radius $R$ for $(P)$.  Suppose that there exists $\delta>0$ such that $R(t)\geq \delta$.
 Moreover suppose that  for  all  $(x_*(t),u)$ such that  $(t,x_*(t), \phi(x_*(t),u))$ is an admissible cluster point of $x_*$ in the sense of \cite[Definition 4.1]{anli},  the WBCQ holds:
\begin{eqnarray*}
\left \{ \begin{array}{l}  (\alpha,0)\in \partial\langle \lambda, \Phi\rangle(x_*(t),u)+\{0\}\times N_{U}(u),\\
\lambda\in N_{\Omega}(\Phi(x_*(t),u))\end{array} \right. \Longrightarrow \alpha=0
\end{eqnarray*}
and the mapping $M$ defined as in (\ref{perturb3.1}) is calm at $(0, x_*(t),u)$. Then the necessary optimality conditions of Proposition \ref{Propnew} holds as stated with some radius $\eta\in (0,R)$. Moreover if $\dot{x}_*(\cdot)$ is continuous, then the WBCQ and the calmness condition are only required to hold along $(x_*(t),u_*(t)).$
\end{proposition}
{\bf Proof of Theorem \ref{thm4.3}}.
The proof is similar to the one in Theorem \ref{thm4.2new}. The only difference is that instead of using Proposition \ref{Propnew}, we use Proposition \ref{Propnew1}. The last statement of Theorem \ref{thm4.3} follows from Theorem \ref{New}.   \hfill $\Box$

\section{Optimal control problems with implicit control systems}
The main purpose of this section is to derive necessary optimality conditions for problem $(P_{ICS})$.
As commented in Section 1,
 we can transform $(P_{ICS})$ into the equivalent problem $(P_{ECS})$ by introducing a vector variable $v(t):=\dot{x}(t)$.
%\begin{eqnarray*}
%(P_{ECS})~~~~~~\min &&  J(x,u):=\int_{t_0}^{t_1} F( x(t), u(t),v(t)) dt + f(x(t_0),x(t_1))\nonumber \\
%s.t. && \dot{x}(t)=v(t)\,\quad a.e. \,t\in [t_0,t_1],\nonumber\\
%&& \varphi(x(t), u(t),v(t))\in K_{\varphi}  \quad a.e. \,t\in [t_0,t_1], \nonumber\\
%%&&\psi(x(t), u(t))\in K_{\psi} \quad a.e. \,t\in [t_0,t_1], \label{mconstraint}\\
%&& u(t) \in U \quad a.e. \,t\in [t_0,t_1],\nonumber\\
%&& (x(t_0), x(t_1))\in S.\nonumber
%\end{eqnarray*}
The problem $(P_{ECS})$ is a special case of problem $(P)$ studied in Section 3 with $\phi:=v$.
Unless otherwise specified, in this section we assume that $F, f, \varphi$ are locally Lipschitz continuous, and the sets
$U, K_\varphi$, $S$ are  closed.
It is easy to check that the concept of a local minimum of radius $R(\cdot)$ for the implicit control problem $(P_{ICS})$ defined as in the introduction coincides with the definition of a local minimum of radius $R(\cdot)$ for problem $(P)$.
%For a given admissible pair $(x_{*}, u_{*})$, radius function $R(\cdot)$ and $\varepsilon>0$, we suppose
% that $(x_{*}, u_{*})$ is a local minimum of radius $R(\cdot)$ for $P_{I																								 CS}$.
% %: $(x_{*}, u_{*})$
%%minimizes the value of the cost
%%function $J(x, u)$ over all admissible pairs $(x,u)$ which satisfies
%%$$| x(t)-x_{*}(t)|\leq\varepsilon, \,|(u(t), \dot{x}(t))-(u_{*}(t), \dot{x}_*(t))|\leq R(t) \mbox{ a.e.,} \int_{t_0}^{t_1}|\dot{x}(t)-\dot{x}_{*}(t)|dt\leq \varepsilon.$$
%Define a set-valued map as  the perturbed constrained system:
%\begin{equation}
%{M}_\varphi(\Theta):=\left\{(x,u,v)\in \mathbb{R}^{n_{x}}\times U \times \mathbb{R}^{n_{x}}:
%%\left (\begin{array}{c}
%                                           \varphi(x,u,v)+ \Theta\in K_{\varphi}
%                                  %             \left (\begin{array}{c}K_{\varphi}
%%                                           \\
%%                                          \psi(x,u)\end{array}\right)+\Theta\in
%%                                               \left (\begin{array}{c}K_{\varphi} \\
%%                                               K_{\psi} \end{array}\right)
%\right\}.\label{perturb4.1n}
%\end{equation}
Define
\begin{eqnarray}
&&{S}_\varphi^{\varepsilon,R}(t):=\{(x,u,v)\in {M}_\varphi(0): |x-x_*(t)|\leq \varepsilon, |(u,v)-(u_{*}(t),\dot{x}_{*}(t))|\leq R(t)\},\nonumber \\
&&{C}_\varphi^{\varepsilon,R}:=cl\{(t,x,u,v)\in [t_0,t_1]\times \mathbb{R}^{n_{x}}\times \mathbb{R}^{n_{u}}\times \mathbb{R}^{n_{x}}:  (x,u,v)\in {S}_\varphi^{\varepsilon,R}(t)\} \label{C},
\end{eqnarray}
{where the set-valued map  ${M}_\varphi(\Theta)$ is defined as in  (\ref{perturb4.1n}).}
With these identifications, the following results follow immediately from Theorems \ref{thm4.2new}, \ref{New}, \ref{thm4.3} and the calculus rule for normal cones.
%The main result of this section is the following.
\begin{thm} \label{thm4.1} Let  $(x_*,u_*)$ be a  local minimum of radius $R(\cdot)$ for $(P_{ICS})$.
  Suppose that  there exists $\delta>0$ such that $R(t)\geq \delta$.
 Suppose further that ${C}_\varphi^{\varepsilon, R}$ is compact and for
 %almost every $t$, $(x,u)\in S_*^{\varepsilon,R}(t)$,
 all $(t,x,u,v) \in {C}_\varphi^{\varepsilon, R}$  the WBCQ holds:
\begin{eqnarray*}
\left \{ \begin{array}{l}  (\alpha,0,0)\in \partial \langle \lambda_{\varphi}, \varphi\rangle(x,u,v)+\{0\}\times N_U(u)\times \{0\},
%+\partial^L\langle \lambda_{\psi}, \psi\rangle(x,u)
\\
\lambda_{\varphi}\in N_{K_{\varphi}}(\varphi(x,u,v))
%\lambda_{\psi}\in N_{K_{\psi}}(\psi(x,u)), \beta\in N_{U}^{L}(u)
\end{array} \right. \Longrightarrow \alpha=0
\end{eqnarray*}
and the mapping ${M_\varphi}$ defined as in (\ref{perturb4.1n}) is calm at $(0, x,u,v)$.
Then there exist an arc $p$ and   a number $\lambda _{0}$ in $\{0,1\}$,
%and  a measurable function $\lambda:[t_0,t_1]\rightarrow \mathbb{R}^n$ with
% for almost every $t\in [t_0,t_1]$
satisfying the nontriviality condition
$(\lambda _{0},p(t))\neq0, \forall t\in[t_0,t_1]$,
 the transversality condition
$$(p(t_0),-p(t_1)) \in \lambda_0 \partial f(x_*(t_0),x_*(t_1))+N_S (x_*(t_0),x_*(t_1)),$$
and
the Euler adjoint inclusion for almost every $t$:
\begin{eqnarray}
\lefteqn{(\dot{p}(t),-\mu(t), p(t)) \in
\lambda_0\partial^CF( x_*(t),u_*(t),\dot{x}_*(t))}\nonumber \\
 &&+co\{ \partial \langle \lambda_{\varphi},\varphi\rangle( x_*(t),u_*(t),\dot{x}_*(t))
 %+\langle \lambda_{\psi}, \psi\rangle)
 :\lambda_{\varphi}\in N_{K_{\varphi}} (\varphi(x,u,v))
%\lambda_{\psi}\in N_{K_{\psi}}^{L}(\psi(x,u))
\}, \nonumber
 \label{41}
\end{eqnarray}
where $\mu(\cdot)$ is a measurable function satisfying $\mu(t)\in N_{U}^{C}(u_*(t))$ a.e.,
as well as the Weierstrass condition of radius $R(\cdot)$ for almost every $t$:
\begin{eqnarray*}
\lefteqn{ (x_*(t),u,v)\in {M_\varphi}(0), |(u,v)-(u_*(t),\dot{x}_*(t))| < R(t)\Longrightarrow }\\
&& \langle p(t),v\rangle-  \lambda_{0}F(x_*(t),u,v)\leq  \langle p(t),\dot{x}_*(t) \rangle -\lambda_0F(x_*(t),u_*(t),\dot{x}_*(t)).
\end{eqnarray*}
Moreover if either $K_\varphi
%\times K_\psi
\subseteq \mathbb{R}^{m}_{-}$ or $\varphi$ is strictly differentiable,
%\times \mathbb{R}^{d}_{-}$.
then the Euler adjoint inclusion can be replaced by the one in the explicit multiplier form, i.e., there exists
measurable functions $\lambda_{\varphi}:[t_0,t_1]\rightarrow \mathbb{R}^m_{+}$,
% $\lambda_{\psi}:[t_0,t_1]\rightarrow \mathbb{R}^d_{+}$,
$\mu:[t_0,t_1]\rightarrow \mathbb{R}^{n_{u}}$ with
$\lambda_{\varphi}(t)\in N_{K_{\varphi}}^{C}(\varphi(x_*(t),u_*(t),\dot{x}_*(t)))$,
% $\lambda_{\psi}(t)\in N_{K_{\psi}}^{C}(\psi(x_*(t),u_*(t)))$,
 $\mu(t)\in N_{U}^{C}(u_*(t))$ a.e.
satisfying
$$
(\dot{p}(t),-\mu(t), p(t)) \in \lambda_0\partial^CF( x_*(t),u_*(t),\dot{x}_*(t))+\partial^C \varphi(x_*(t),u_*(t),\dot{x}_*(t))^T\lambda_{\varphi}(t)\mbox { a.e.}.
$$
If $N_U^C(u_*(t))=\{0\}$,  { $K_{\varphi}$ is normally regular} and $\varphi$ is  strictly differentiable, then  the estimate for the  multiplier $\lambda_{\varphi}(t)$ in (\ref{3.9}) also holds, namely,
$$
|\lambda_{\varphi}(t)|\leq \kappa\{k|p(t)|+\lambda_0k^{F}\}\qquad a.e.
$$
for some positive constants $k, \kappa, k^{F} $, where   $k^{F}$ is the Lipschitz coefficients of $F$ on set $D$ defined as in (\ref{setD}) respectively.
Moreover if $u_*(\cdot)$ is continuous, then the WBCQ and the calmness condition are only required to hold along $(x_*(t),u_*(t))$.
In the case of free end point, $\lambda_0$ can be taken as $1$.
\end{thm}

%The Euler adjoint inclusion (\ref{41}) in Theorem \ref{thm4.1} is in an implicit form. However if $\lambda_{\varphi}
%%,\lambda_{\psi}
%$ in the Euler inclusion (\ref{41}) are all nonnegative, it is easy to derive the Euler inclusion in the explicit multiplier form.
%\begin{corollary}  In additions to  the assumptions of Theorems \ref{thm4.1}, suppose that  $K_\varphi
%%\times K_\psi
%\subseteq \mathbb{R}^{k}_{-}$.
%%\times \mathbb{R}^{d}_{-}$.
%Then the Euler adjoint inclusion can be replaced by the one in the explicit multiplier form, i.e., there exists
%measurable functions $\lambda_{\varphi}:[t_0,t_1]\rightarrow \mathbb{R}^k_{+}$,
%% $\lambda_{\psi}:[t_0,t_1]\rightarrow \mathbb{R}^d_{+}$,
%$\mu:[t_0,t_1]\rightarrow \mathbb{R}^{n_{u}}$ with
%$\lambda_{\varphi}(t)\in N_{K_{\varphi}}^{C}(\varphi(x_*(t),u_*(t),\dot{x}_*(t)))$,
%% $\lambda_{\psi}(t)\in N_{K_{\psi}}^{C}(\psi(x_*(t),u_*(t)))$,
% $\mu(t)\in N_{U}^{C}(u_*(t))$ a.e.
%satisfying
%\begin{eqnarray*}
%(\dot{p}(t),-\mu(t), p(t)) \in \lambda_0&&\hspace{-0.6cm}\partial^CF( x_*(t),u_*(t),\dot{x}_*(t))+\partial^C \varphi(x_*(t),u_*(t),\dot{x}_*(t))^T\lambda_{\varphi}(t)\\
%%&& + \partial^C \psi(x_*(t),u_*(t))^T\lambda_{\psi}(t),
%&&  \mbox { a.e. } t\in [t_0,t_1].
%\end{eqnarray*}
%\end{corollary}

A special case of the optimal control of implicit systems is the following problem
\begin{eqnarray*}
(P_{DAE})~~~\min &&  {J(x,u):=\int_{t_0}^{t_1} F(x(t), u(t),\dot{x}(t)) dt + f(x(t_0),x(t_1)),}\nonumber \\
s.t. && \varphi (x(t), u(t),\dot{x}(t))=0{,}\\
&& u(t) \in U \quad a.e. \,t\in [t_0,t_1],\nonumber\\
&& (x(t_0), x(t_1))\in S.\nonumber
\end{eqnarray*}
{This problem was studied in \cite[Section 6]{cp} with a time dependent control set $U(t)$.}
Applying Theorem \ref{thm4.1} with $K_\varphi=\{0\}$, we immediately have the following result.
\begin{corollary} \label{thm5.1new} Let  $(x_*,u_*)$ be a  local minimum of radius $R(\cdot)$ for $(P_{DAE})$.
  Suppose that  there exists $\delta>0$ such that $R(t)\geq \delta$.
 Suppose further that
 ${C}_\varphi^{\varepsilon,R}$ as defined in (\ref{C}) with $K_\varphi=\{0\}$  is compact and for
 %almost every $t$, $(x,u)\in S_*^{\varepsilon,R}(t)$,
 all $(t,x,u,v) \in {C}_\varphi^{\varepsilon, R}$  the WBCQ holds:
\begin{eqnarray*}
 \lambda_{\varphi}\in \mathbb{R}^m,\, (\alpha,0,0)\in \partial \langle \lambda_{\varphi}, \varphi\rangle(x,u,v)+\{0\}\times N_U (u)\times \{0\}
\Longrightarrow \alpha=0
\end{eqnarray*}
and the mapping ${M}_\varphi$ as defined in (\ref{perturb4.1n}) with $K_\varphi=\{0\}$ is calm at $(0, x,u,v)$.
Then there exist an arc $p$ and   a number $\lambda _{0}$ in $\{0,1\}$,
%and  a measurable function $\lambda:[t_0,t_1]\rightarrow \mathbb{R}^n$ with
% for almost every $t\in [t_0,t_1]$
satisfying the nontriviality condition
$(\lambda _{0},p(t))\neq0, \forall t\in[t_0,t_1]$,
 the transversality condition
$$(p(t_0),-p(t_1)) \in \lambda_0 \partial f(x_*(t_0),x_*(t_1))+N_S (x_*(t_0),x_*(t_1)),$$
and
the Euler adjoint inclusion for almost every $t$:
\begin{eqnarray*}
\lefteqn{(\dot{p}(t),-\mu(t), p(t)) \in
\lambda_0\partial^CF( x_*(t),u_*(t),\dot{x}_*(t))}\nonumber \\
 &&+co\{ \partial \langle \lambda_{\varphi},\varphi\rangle( x_*(t),u_*(t),\dot{x}_*(t))
 %+\langle \lambda_{\psi}, \psi\rangle)
 :\lambda_{\varphi}\in \mathbb{R}^m
%\lambda_{\psi}\in N_{K_{\psi}}^{L}(\psi(x,u))
\},
\end{eqnarray*}
where $\mu(\cdot)$ is a measurable function satisfying $\mu(t)\in N_{U}^{C}(u_*(t))$ a.e.,
as well as the Weierstrass condition of radius $R(\cdot)$ for almost every $t$:\begin{eqnarray*}
\lefteqn{ u\in U, \varphi(x_*(t), u, v)=0, |(u,v)-(u_*(t),\dot{x}_*(t))| < R(t)\Longrightarrow }\\
&& \langle p(t),v\rangle-  \lambda_{0}F(x_*(t),u,v)\leq  \langle p(t),\dot{x}_*(t) \rangle -\lambda_0F(x_*(t),u_*(t),\dot{x}_*(t)).
\end{eqnarray*}
Suppose further that $\varphi$ is {strictly} differentiable, then  the Euler adjoint inclusion can be expressed in the explicit form: there exists
measurable functions $\lambda_{\varphi}:[t_0,t_1]\rightarrow \mathbb{R}^m$,
%$\lambda_{\psi}:[t_0,t_1]\rightarrow \mathbb{R}^d$,
$\mu:[t_0,t_1]\rightarrow \mathbb{R}^{n_{u}}$ with
%$\lambda_{\varphi}(t)\in N_{K_{\varphi}}^{C}(\varphi(x_*(t),u_*(t),\dot{x}_*(t)))$,
%$\lambda_{\psi}(t)\in N_{K_{\psi}}^{C}(\psi(x_*(t),u_*(t)))$,
$\mu(t)\in N_{U}^{C}(u_*(t))$ a.e. such that
$$
(\dot{p}(t),-\mu(t), p(t)) \in
 \lambda_0\partial^CF( x_*(t),u_*(t),\dot{x}_*(t))+\nabla \varphi(x_*(t),u_*(t),\dot{x}_*(t))^T\lambda_{\varphi}(t)
%+ \nabla \psi(x_*(t),u_*(t))^T\lambda_{\psi}(t),
\mbox { a.e.}.
$$
If $N_U^C(u_*(t))=\{0\}$, then  the estimate for the  multiplier $\lambda_{\varphi}(t)$ in (\ref{3.9}) also holds:
$$
|\lambda_{\varphi}(t)|\leq \kappa\{k|p(t)|+\lambda_0k^{F}\}\qquad a.e.
$$
for some positive constants $k, \kappa, k^{F} $, where   $k^{F}$ is the Lipschitz coefficients of $F$ on set $D$ defined as in (\ref{setD}) respectively.
Moreover if $u_*(\cdot)$ is continuous, then the WBCQ and the calmness condition are only required to hold along $(x_*(t),u_*(t),\dot{x}_*(t))$. In the case of free end point, $\lambda_0$ can be taken as $1$.

\end{corollary}

 Note that in   \cite[Theorem 6.1 and Corollary 6.2]{cp}, a similar result is obtained. Their results allow for the dynamic system to be nonautonomous but they require the calibrated constraint qualification or MFC to hold which are stronger than WBCQ+calmness.
% \textcolor{red}{What's more, the condition MFC also need the smoothness of $\varphi$ while no smoothness of $\varphi$ is required in our result unless we wish to get the Euler adjoint inclusion in the explicit form and the estimate for the multiplier.
%}  \textcolor{blue}{ Please delete the red-colored sentence!}

\section{Optimal control of semi-explicit DAEs}
In this section we
consider  the following optimal control problem of {semi-explicit} DAEs:
\begin{eqnarray*}
(P_{seDAE})~~~\min &&  J(x,y,u):=\int_{t_0}^{t_1} F(x(t),y(t), u(t)) dt + f(x(t_0),x(t_1)),\nonumber \\
s.t. && \dot{x}(t)=\phi(x(t),y(t), u(t)) \quad a.e. \,t\in [t_0,t_1],\nonumber\\
&&  0=h(x(t),y(t), u(t))\quad a.e. \,t\in [t_0,t_1],\nonumber\\
&& u(t) \in U \quad a.e. \,t\in [t_0,t_1],\nonumber\\
&& (x(t_0), x(t_1))\in S,\nonumber
\end{eqnarray*}
where $F:\mathbb{R}^{n_{x}}\times \mathbb{R}^{n_{y}}\times \mathbb{R}^{n_{u}}\rightarrow \mathbb{R}$,
$\phi :\mathbb{R}^{n_{x}}\times \mathbb{R}^{n_{y}}\times \mathbb{R}^{n_{u}}\rightarrow \mathbb{R}^{n_{x}}$,
$h :\mathbb{R}^{n_{x}}\times \mathbb{R}^{n_{y}}\times \mathbb{R}^{n_{u}}\rightarrow \mathbb{R}^{n_y}$, the others are the same as in $(P)$. In this section, unless otherwise specified we assume that $F, f, \phi, h$ are locally Lipschitz continuous.

%{Necessary optimality conditions for such a model has been studied in the literature (see e.g.  \cite{pinhovinter,roub,igor}).}
%$\psi:\mathbb{R}^{n_{x}}\times \mathbb{R}^{n_{u}}\rightarrow  \mathbb{R}^d$;
%$K_{\varphi}\subseteq \mathbb{R}^k$, $K_{\psi}\subseteq\mathbb{R}^d$, $U\subseteq \mathbb{R}^{n_{u}}$, $S\subseteq\mathbb{R}^{n_{x}}\times \mathbb{R}^{n_{x}}$.
The dynamic is said to have{``index $k$''} if one needs to differentiate the algebraic part $(k-1)$-times in time to get the underlying system of ODE \cite{Griepentrog}.
The main restriction on the  necessary optimality condition of the optimal control problem of semi-explicit DAEs is the assumption that the dynamics have ``index one'' (see e.g.\cite{pinhovinter,cp,igor}), i.e., the Jacobian matrix $\nabla_y h (x_{*}(t),y_{*}(t), u_{*}(t))$ has full  rank, or  equivalently  %\textcolor{red}{The following does not make sense since $h_y$ is not a square matrix.}
$$\det \nabla_y h (x_{*}(t),y_{*}(t), u_{*}(t))\neq 0.$$
In the index one case, by using the implicit function theory, the variable $y(t)$ can be solved locally and hence the system behaves like an ODE.
Derivation of optimality conditions for higher index problems is a challenging area.

%For simplicity of the exposition, we omit the equality and inequality constraints.
We take two approaches to study the problem. In the first approach  we treat $y$ as a control and explore the consequences of Theorem \ref{thm4.2new} and in the second approach we treat $y$ as a state and explore the consequences of Corollary \ref{thm5.1new}. Both approaches allow us to derive necessary optimality conditions  without the assumption that the problem is of index one. Such approaches have also been taken in \cite{igor} to specialize the results of \cite{cp} to the control of semi-explicit DAEs. But their results can only be applied to problem of index one.

%For a given admissible pair $(x_{*}, u_{*})$, radius function $R(\cdot)$, and $\varepsilon>0$,
If we treat $y$ as a control, then both $u(\cdot)$ and $y(\cdot)$ are  measurable functions on $[t_0,t_1]$ such that $u(t)\in U$ for almost every $t \in [t_0,t_1]$. The state  corresponding to a given control $(u(\cdot), y(\cdot))$, refers to an absolutely continuous function  $x(\cdot)$ which together with $u(\cdot), y(\cdot)$ satisfying all the constraints of the problem ($P_{seDAE}$). We call such a pair $(x(\cdot), y(\cdot), u(\cdot))$ an admissible pair.   Let { $R:[t_0,t_1]\rightarrow (0,+\infty]$ } be a radius function. We say that $(x_{*},y_{*}, u_{*})$ is a local minimum of radius $R(\cdot)$ for $(P_{seDAE})$ if it
minimizes the value of the cost
function $J(x, y, u)$ over all admissible pairs $(x, y, u)$ which satisfies
$$| x(t)-x_{*}(t)|\leq\varepsilon, \,|(y(t), u(t))-(y_{*}(t), u_*(t))|\leq R(t) \mbox{ a.e.,} \int_{t_0}^{t_1}|\dot{x}(t)-\dot{x}_{*}(t)|dt\leq \varepsilon.$$
Define a set-valued map as  the perturbed constrained system:
\begin{eqnarray}\label{Mh}
M_h(\Theta):=\left\{(x,y,u)\in \mathbb{R}^{n_{x}}\times \mathbb{R}^{n_{y}}\times U :   h(x,y,u)+ \Theta=0 \right\}
\end{eqnarray}
and
\begin{eqnarray*}
&& {S}_h^{\varepsilon,R}(t):=\{(x,y,u)\in M_h(0): |x-x_*(t)|\leq \varepsilon, |(y,u)-(y_{*}(t), u_*(t))|\leq R(t)\},\\
&& {C}_h^{\varepsilon,R}:=cl\{(t,x,y,u)\in [t_0,t_1]\times \mathbb{R}^{n_{x}}\times \mathbb{R}^{n_{y}}\times U:  (x,y,u)\in  {S}_h^{\varepsilon,R}(t)\}.
\end{eqnarray*}

A simple application of Theorem \ref{thm4.2new} yields the following results.
% which can be applied to a problem with any index.
\begin{thm}\label{thm5.1}  Let  $(x_*,y_*,u_*)$ be a  local minimum of radius $R(\cdot)$ for $(P_{seDAE})$.
  Suppose that ${C}_h^{\varepsilon,R}$ is compact, and there exists $\delta>0$ such that $R(t)\geq \delta$.
 Suppose further that, for
 %almost every $t$, $(x,u)\in S_*^{\varepsilon,R}(t)$,
 all $(t,x,y,u) \in {C}_h^{\varepsilon, R}$  the WBCQ holds:
\begin{eqnarray}\label{WBCQw}
 \lambda\in \mathbb{R}^{n_y},  (\alpha,0,0)\in \partial \langle \lambda, h\rangle(x,y,u)+\{(0,0)\}\times N_U (u)
\Longrightarrow \alpha=0,
\end{eqnarray}
and the mapping $M_h$ is calm at $(0, x,y,u)$.
Then there exist an arc $p$ and   a number $\lambda _{0}$ in $\{0,1\}$,
%and  a measurable function $\lambda:[t_0,t_1]\rightarrow \mathbb{R}^n$ with
% for almost every $t\in [t_0,t_1]$
satisfying the nontriviality condition
$(\lambda _{0},p(t))\neq0, \forall t\in[t_0,t_1]$,
 the transversality condition
$$(p(t_0),-p(t_1)) \in \lambda_0 \partial f(x_*(t_0),x_*(t_1))+N_S (x_*(t_0),x_*(t_1)),$$
and
the Euler adjoint inclusion for almost every $t$:
\begin{eqnarray}
\lefteqn{(\dot{p}(t), 0,-\mu(t)) \in
\partial^C \{\langle -p(t), \phi\rangle+\lambda_0F\} ( x_*(t),y_*(t),u_*(t))}\nonumber \\
&& +co\{ \partial \langle  \lambda,  h( x_*(t),y_*(t),u_*(t))\rangle: \lambda\in \mathbb{R}^{n_y}\},\nonumber \label{EulerN}
\end{eqnarray}
where $\mu(\cdot)$ is a measurable function satisfying $\mu(t)\in N_{U}^{C}(u_*(t))$ a.e.,
as well as the Weierstrass condition of radius $R(\cdot)$ for almost every $t$:
\begin{eqnarray*}
&&u\in U, \,\,h(x_*(t),y,u)=0,\,\, \,  |(y,u)-(y_{*}(t), u_*(t))|< R(t)\Longrightarrow\\
&&\langle p(t),\phi(x_*(t),y,u) \rangle-\lambda_{0}F(x_*(t),y,u) \leq \langle p(t), \phi (x_*(t),y_*(t),u_*(t)) \rangle \\
&& \qquad \qquad -\lambda_{0}F(x_*(t),y_*(t),u_*(t)).
\end{eqnarray*}
Moreover if we assume further that  $h$ is strictly differentiable,
then the Euler adjoint inclusion can be replaced by the one in the explicit multiplier form, i.e., there exist
measurable functions $\lambda_{h}:[t_0,t_1]\rightarrow \mathbb{R}^{n_y}$, $\mu:[t_0,t_1]\rightarrow \mathbb{R}^{n_{u}}$ with $\mu(t)\in N_{U}^{C}(u_*(t))$ a.e.
satisfying
\begin{eqnarray*}
\lefteqn{(\dot{p}(t),0,-\mu(t)) \in
 \partial^C \{\langle -p(t), \phi\rangle+\lambda_0F\} ( x_*(t),y_*(t),u_*(t))}\nonumber\\
&& + \nabla h(x_*(t),y_*(t),u_*(t))^T\lambda_{h}(t), \mbox { a.e. } t\in [t_0,t_1].
\end{eqnarray*}
If $N_U(u_*(t))={\{0\}}$, then  the estimate for the  multiplier $\lambda_{h}(t)$ in (\ref{3.9}) also holds:
$$
|{\lambda_{h}}(t)|\leq \kappa\{k|p(t)|+\lambda_0k^{F}\}\qquad a.e.
$$
for some positive constants $k, \kappa, k^{F} $, where   $k^{F}$ is the Lipschitz coefficients of $F$ on set $D$ defined as in (\ref{setD}) respectively.
Moreover if $u_*(\cdot)$ is continuous, then the WBCQ and the calmness condition are only required to hold along $(x_*(t),y_*(t),u_*(t))$. In the case of free end point, $\lambda_0$ can be taken as $1$.
\end{thm}
%\begin{proof}ddddddd
%\end{proof}

Note that our necessary optimality condition is not the so-called strong maximum principle as in \cite[Theorem 3.1]{pinhovinter}. It was shown in  \cite{pinhovinter} by using the following example that  that a strong maximum principle may not hold if the velocity set is nonconvex. But the conclusion of our necessary optimality condition is more than just weak maximum principle as in \cite[Theorem 3.2]{pinhovinter}. In fact only the nontriviality condition, the transversality condition and  the Euler adjoint inclusion alone constitute the weak maximum principle, let alone the extra Weiersrass condition. A consequence is that we derive the weak maximum principle under the WBCQ plus calmness condition which allows application to problems with index higher than one.

\begin{example}\label{ex5.1}\cite{pinhovinter}.
\begin{eqnarray*}
~~~\min &&  -x(1)\nonumber \\
s.t. && \dot{x}(t)=(u(t)-y(t))^2\quad a.e. \,t\in [0, 1],\nonumber\\
&&  0=u(t)-y(t)\quad a.e. \,t\in [0, 1],\nonumber\\
&& u(t) \in [-1,1]{,}\\
&& x(0)=0.
\end{eqnarray*}
\end{example}
In this example, the function $h$ is independent of $x$ and is affine. In fact  if $h$ is independent of $x$ and is affine, by  \cite[Proposition 2.2]{anli}, $M_h$ is calm. Consequently  the WBCQ plus calmness condition holds automatically.
Then the following results follow from Theorem \ref{thm5.1}.
\begin{corollary}\label{coro5.1}  Let  $(x_*,y_*,u_*)$ be a  local minimum of radius $R(\cdot)$ for $(P_{seDAE})$.
  Suppose that $F, f, \phi$ are locally Lipschitz continuous,  $h$ is independent of  the variable $x$ and is affine   and $U$ is a union of finitely many polyhedral sets. {Suppose further that} ${C}_h^{\varepsilon,R}$ is compact, and there exists $\delta>0$ such that $R(t)\geq \delta$.
 % $\tilde{C}_*^{\varepsilon,R}$ is compact, and there exists $\delta>0$ such that $R(t)\geq \delta$.
% Suppose further that, for
% all $(t,x,y,u) \in \tilde{C}_*^{\varepsilon, R}$  the WBCQ holds:
%\begin{eqnarray*}
% \lambda\in \mathbb{R}^{k}, \mu\in N_{U}^{L}(u),  (\alpha,0,-\mu)\in \partial^L\langle \lambda, h\rangle(x,y,u)
%\Longrightarrow \alpha=0,
%\end{eqnarray*}
%and the mapping $\acute{M}$ is calm at $(0, x,y,u)$.
Then the conclusions of Theorem \ref{thm5.1} hold with the  explicit Euler adjoint inclusion
\begin{eqnarray*}
\lefteqn{(\dot{p}(t),0,-\mu(t)) \in \partial^C \{\langle -p(t), \phi\rangle+\lambda_0F\} ( x_*(t),y_*(t),u_*(t))}\nonumber\\
&& +\nabla h(x_*(t),y_*(t),u_*(t))^T\lambda_{h}(t), \mbox { a.e. } t\in [t_0,t_1].
\end{eqnarray*}
If $N_U^C(u_*(t))=\{0\}$, then  the estimate for the  multiplier $\lambda_{h}(t)$ in (\ref{3.9}) also holds:
$$
|{\lambda_{h}}(t)|\leq \kappa\{k|p(t)|+\lambda_0k^{F}\}\qquad a.e.
$$
for some positive constants $k, \kappa, k^{F} $, where   $k^{F}$ is the Lipschitz coefficients of $F$ on set $D$ defined as in (\ref{setD}) respectively.
\end{corollary}
 Taking $\varepsilon >0$ to be finite and $R(t)=\infty$, it is obvious that $(x_*,y_*,u_*)=(0,0,0)$ is a local minimum of radius $R$ for the problem in Example \ref{ex5.1},  the set $$C_h^{\varepsilon, R}:=\{(t,x,y,u)\in [0,1]\times R\times R\times [-1,1]: y=u, |x|\leq \varepsilon\}$$
is compact. Hence all assumptions in Corollary \ref{coro5.1} holds.  Since it is a free end-point problem, $\lambda_0=1$. It is easy to show that all conditions of the necessary optimality conditions hold with $p(t)\equiv1, \lambda_h(t)\equiv0$.

Now we take the second approach by considering $z=(x,y)$ as the state variable.  We consider the problem $P_{seDAE}$ as the following implicit control problem:
\begin{eqnarray*}
(P_{IDAE})~~~\min &&  J(z,u):=\int_{t_0}^{t_1} F(z(t), u(t)) dt + f(x(t_0),x(t_1)),\nonumber \\
s.t. && \varphi (z(t), u(t),\dot{z}(t))=0{,}\\
&& u(t) \in U \quad a.e. \,t\in [t_0,t_1],\nonumber\\
&& (x(t_0), x(t_1))\in S {,}\nonumber
\end{eqnarray*}
 with $z=(x,y)$ and
$$\varphi(z,u,v):=(\phi(z,u)-v_1,h(z,u))^{T},\,\, v:=(v_1,0)$$
$v_1\in \mathbb{R}^{n_x}$ and apply Corollary \ref{thm5.1new}.
%\begin{eqnarray*}\varphi(z,u,v):=\left(
%                                                  \begin{array}{cc}
%                                                    I & 0 \\
%                                                    0 & 0 \\
%                                                  \end{array}
%                                                \right)
%v-\left (\begin{array}{c}
%                                           \phi(z,u) \\
%                                             h(z,u)\end{array}\right)
%                                             \end{eqnarray*}
%and  $K_{\varphi}:=\{0\}$.  That is, we transform the explicite DAE into an implicit DAE as follows:
%Similar as $(P_{ECS})$, we can  tranform $(P_{DAE})$ into the following equivalent form:
%\begin{eqnarray*}
%(P_{IDAE})~~~~~~\min &&  J(z,u):=\int_{t_0}^{t_1} F(z(t), u(t)) dt + f(x(t_0),x(t_1)),\nonumber \\
%s.t.
%&& \varphi(z(t), u(t),\dot{z}(t))=0 \quad a.e. \,t\in [t_0,t_1], \nonumber\\
%%&&\psi(z(t), u(t))\in K_{\psi} \quad a.e. \,t\in [t_0,t_1], \label{mconstraint}\\
%&& u(t) \in U \quad a.e. \,t\in [t_0,t_1],\nonumber\\
%&& (x(t_0), x(t_1))\in S.\nonumber
%\end{eqnarray*}
The state  corresponding to a given control $u(\cdot)$, refers to an absolutely continuous function  $(x(\cdot),y(\cdot))$ which together with $u(\cdot)$ satisfying all conditions in $(P_{IDAE})$.   Let { $R:[t_0,t_1]\rightarrow (0,+\infty]$ }  be a radius function. We say that $(x_{*},y_{*}, u_{*})$ is a local minimum of radius $R(\cdot)$ for $P_{IDAE}$ if it
minimizes  the value of the cost
function $J(x,y, u)$ over all admissible pairs $(x,y,u)$ which satisfies
\begin{eqnarray*}
&& | (x(t),y(t))-(x_*(t),y_*(t))|\leq\varepsilon, \,|(u(t), \dot{x}(t),\dot{y}(t))-(u_{*}(t), \dot{x}_*(t) ,  \dot{y}_*(t))|\leq R(t) \mbox{ a.e.,} \\
&& \int_{t_0}^{t_1}|(\dot{x}(t),\dot{y}(t))-(\dot{x}_*(t) ,  \dot{y}_*(t))|dt\leq \varepsilon.
\end{eqnarray*}

%Moreover, for any $(\lambda,\upsilon,\mu,\nu)\in \Re^{l_1}\times \Re^{l_2} \times \Re^{l} \times \Re^l$, we denote
%\[
%\Upsilon(t,x,u;\lambda,\upsilon,\mu,\nu):=g(x,u)^\top\lambda+h(x,u)^\top\upsilon-G(x,u,v)^\top\mu-H(x,u,v)^\top\nu.
%\]

%\begin{equation}
%\hat{M}(y_1,y_2):=\{(x,u)\in \Re^n\times U: g(x,u)+y_1\leq 0,\, h(x,u)+y_2=0\}. \label{neweqn}
%\end{equation}
%For any given $\varepsilon>0$ and a given radius function $R(\cdot)$, define
%\begin{eqnarray*}
%&&S_*^{\varepsilon,R}(t):=\{(x,u)\in \hat{M}(0): |x-x_*(t)|\leq \varepsilon, |u-u_*(t)|\leq R(t)\},\\
%&&\tilde{C}_*^{\varepsilon,R}:=cl\{(t,x,u)\in [t_0,t_1]\times \Re^{n_{x}}\times \Re^{n_{u}}:  (x,u)\in S_*^{\varepsilon,R}(t)\}.
%\end{eqnarray*}

Let $z_*:=(x_*,y_*)$. Define a set-valued map as  the perturbed constrained system:
\begin{equation}\label{perturb4.1new}
M_{\varphi}(\Theta):=\left\{(x,y,u,v)\in \mathbb{R}^{n_{x}}\times \mathbb{R}^{n_{y}}\times U \times\mathbb{R}^{n_{x}+n_{y}}: \varphi(x,y,u,v)+ \Theta=0\right\},
\end{equation}
and
\begin{eqnarray*}
&&S_{\varphi}^{\varepsilon,R}(t):=\{(z,u,v)\in M_{\varphi}(0): |z-z_*(t)|\leq \varepsilon, |(u,v)-(u_{*}(t),  \dot{z}_*(t))|\leq R(t)\},\\
&&C_{\varphi}^{\varepsilon,R}:=cl\{(t,z,u,v)\in [t_0,t_1]\times \mathbb{R}^{n_{x}}\times \mathbb{R}^{n_{y}} \times \mathbb{R}^{n_{u}}\times \mathbb{R}^{n_{x}}\times \mathbb{R}^{n_{y}} :  (z,u,v)\in S_{\varphi}^{\varepsilon,R}(t)\}.
\end{eqnarray*}
With these identifications, we can apply Corollary \ref{thm5.1new} and  obtain the results as follows.
\begin{thm} \label{thm5.2} Let  $(x_*,y_*,u_*)$ be a  local minimum of radius $R(\cdot)$ for $(P_{seDAE})$ in the above sense.
  Suppose that  there exists $\delta>0$ such that $R(t)\geq \delta$.
 Suppose further that $C_{\varphi}^{\varepsilon,R}$ is compact and for
 %almost every $t$, $(x,u)\in S_*^{\varepsilon,R}(t)$,
 all $(t,z,u,v) \in C_{\varphi}^{\varepsilon,R}$  the WBCQ holds:
\begin{eqnarray}
 \begin{array}{l}  {\lambda \in  \mathbb{R}^{n_y},}  (\alpha_1,\alpha_2,0)\in \partial \langle\lambda, h\rangle (x,y,u)+\{(0,0)\}\times N_U (u)
%\lambda_{\psi}\in N_{K_{\psi}}^{L}(\psi(x,u)), \beta\in N_{U}^{L}(u)
\end{array}  \Longrightarrow \alpha_1=0 , \alpha_2=0\label{WBCQThm5.2}
\end{eqnarray}
and the mapping $M_{\varphi}$  defined  {as}  in (\ref{perturb4.1new})  is calm at $(0, x,y,u,v)$.
Then there exist an arc $p$ and   a number $\lambda _{0}$ in $\{0,1\}$,
%and  a measurable function $\lambda:[t_0,t_1]\rightarrow \mathbb{R}^n$ with
% for almost every $t\in [t_0,t_1]$
satisfying the nontriviality condition
$(\lambda _{0},p(t))\neq0, \forall t\in[t_0,t_1]$,
 the transversality condition
$$(p(t_0),-p(t_1)) \in \lambda_0 \partial f(x_*(t_0),x_*(t_1))+N_S (x_*(t_0),x_*(t_1)),$$
and
the Euler adjoint inclusion for almost every $t$:
\begin{eqnarray*}
\lefteqn{(\dot{p}(t),0,-\mu(t)) \in\lambda_0\partial^CF(x_*(t),y_*(t),u_*(t))}\nonumber \\
&&+co\{ \partial  (\langle \lambda_\phi,\phi\rangle+\langle \lambda_h,h\rangle)(x_*(t),y_*(t),u_*(t)):  \lambda_\phi \in\mathbb{R}^{n_x},  \lambda_h\in \mathbb{R}^{n_y}\},
\end{eqnarray*}
where $\mu(\cdot)$ is a measurable function satisfying $\mu(t)\in N_{U}^{C}(u_*(t))$ a.e.,
as well as the Weierstrass condition of radius $R(\cdot)$ for almost every $t$:
\begin{eqnarray*}
&& \phi(x_*(t),y_*(t),u)-w=0, h(x_*(t),y_*(t),u)=0,|(u,v)-(u_*(t),\dot{z}_*(t)| < R(t), \\
&&u\in U\Longrightarrow\langle p(t),v-\dot{z}_*(t)\rangle
\leq \lambda_{0}(F(x_*(t),y_*(t),u)-F(x_*(t),y_*(t),u_*(t)){)}.
\end{eqnarray*}
Suppose further that $\phi, h$ are {strictly} differentiable, then  the Euler adjoint inclusion can be expressed in the explicit form: there exist
measurable functions $ \lambda_h:[t_0,t_1]\rightarrow \mathbb{R}^{n_y}$,
$\mu:[t_0,t_1]\rightarrow \mathbb{R}^{n_{u}}$ with
%$\lambda_{\varphi}(t)\in N_{K_{\varphi}}^{C}(\varphi(x_*(t),u_*(t),\dot{x}_*(t)))$,
%$\lambda_{\psi}(t)\in N_{K_{\psi}}^{C}(\psi(x_*(t),u_*(t)))$,
$\mu(t)\in N_{U}^{C}(u_*(t))$ a.e. such that
\begin{eqnarray*}
\lefteqn{(\dot{p}(t),0,-\mu(t)) \in
 \lambda_0\partial^CF(x_*(t),y_*(t),u_*(t))}\nonumber \\
&&+\nabla \phi(x_*(t),y_*(t),u_*(t))^Tp(t)+\nabla h(x_*(t),y_*(t),u_*(t))^T\lambda_h (t).
\end{eqnarray*}
In the case of free end point, $\lambda_0$ can be taken as $1$.
\end{thm}
\begin{proof} By Corollary \ref{thm5.1new},  if for any $(t, z, u,v) \in C_{\varphi}^{\varepsilon,R}$, the WBCQ holds:
\begin{eqnarray}
\left \{ \begin{array}{l}  (\alpha,0,0)\in \partial_{z,u}\{\langle\lambda_{1}, \phi\rangle +\langle\lambda_{2}, h \rangle\} (z,u)\times \{0\}
\\
\qquad +\{(0,0)\}\times N_U (u)\times\{-\lambda_1\},
%+\partial^L\langle \lambda_{\psi}, \psi\rangle(x,u)
\\
\lambda_{1} \in  \mathbb{R}^{n_{x}},\,\lambda_{2} \in  \mathbb{R}^{n_y}
%\lambda_{\psi}\in N_{K_{\psi}}^{L}(\psi(x,u)), \beta\in N_{U}^{L}(u)
\end{array} \right. \Longrightarrow \alpha=0 \label{WBCQe}
\end{eqnarray}
and  the mapping $M_{\varphi}$ is calm at $(0, x,y,u,v)$, then there exist arcs $p_x,p_y$ and $\lambda_0\in \{0,1\}$, satisfying the nontriviality condition
$(\lambda _{0},p_x(t),p_y(t))\neq0, \forall t\in[t_0,t_1]$,
 the transversality condition
$$(p_x(t_0),-p_x(t_1)) \in \lambda_0 \partial f(x_*(t_0),x_*(t_1))+N_S(x_*(t_0),x_*(t_1)); \,\,p_y(t_0)=0, p_y(t_1)=0,$$
and
the Euler adjoint inclusion for almost every $t$:
\begin{eqnarray*}
\lefteqn{(\dot{p}_{x}(t),\dot{p}_{y}(t),-\mu(t), p_{x}(t),p_{y}(t)) \in}\\
&&  \lambda_0\partial^CF(x_*(t),y_*(t),u_*(t))\times \{(0,0)\}\nonumber \\
&&+co\{ \partial(\langle \lambda_{\phi},\phi \rangle+\langle \lambda_h, h\rangle )( x_*(t),u_*(t),\dot{x}_*(t))\times \{(-\lambda_\phi, 0)\}: \lambda_\phi \in \mathbb{R}^{n_x}, \lambda_h \in \mathbb{R}^{n_y}\}
\end{eqnarray*}
where $\mu(\cdot)$ is a measurable function satisfying $\mu(t)\in N_{U}^{C}(u_*(t))$ a.e.,
as well as the Weierstrass condition of radius $R(\cdot)$ for almost every $t$:
\begin{eqnarray*}																																		
&& \phi(x_*(t),y_*(t),u)-v_1=0, h(x_*(t),y_*(t),u)=0,|(u,v)-(u_*(t),\dot{z}_*(t))| < R(t), \\
&& u\in U\Longrightarrow\langle p(t),v-\dot{z}_*(t)\rangle
\leq \lambda_{0}(F(x_*(t),y_*(t),u)-F(x_*(t),y_*(t),u_*(t))).
\end{eqnarray*}
Suppose further that $\phi, h$ are {strictly} differentiable, then  the Euler adjoint inclusion can be expressed in the explicit form: there exist
measurable functions $\lambda_{\phi}:[t_0,t_1]\rightarrow \mathbb{R}^{n_{x}}$, $ \lambda_h:[t_0,t_1]\rightarrow \mathbb{R}^{n_y}$,
$\mu:[t_0,t_1]\rightarrow \mathbb{R}^{n_{u}}$ with
%$\lambda_{\varphi}(t)\in N_{K_{\varphi}}^{C}(\varphi(x_*(t),u_*(t),\dot{x}_*(t)))$,
%$\lambda_{\psi}(t)\in N_{K_{\psi}}^{C}(\psi(x_*(t),u_*(t)))$,
$\mu(t)\in N_{U}^{C}(u_*(t))$ a.e. such that
\begin{eqnarray*}
\lefteqn{(\dot{p}_{x}(t),\dot{p}_{y}(t),-\mu(t), p_{x}(t),p_{y}(t)) \in}\\
&&  \lambda_0\partial^CF(x_*(t),y_*(t),u_*(t))\times \{(0,0)\}\nonumber \\
&&+(\nabla  \phi(x_*(t),y_*(t),u_*(t))^T\lambda_\phi(t) +\nabla  h(x_*(t),y_*(t),u_*(t))^T\lambda_h(t))\times \{\lambda_\phi(t),0\}, a.e.
\end{eqnarray*}

It is easy to see that the WBCQ  (\ref{WBCQe}) is equivalent to the WBCQ (\ref{WBCQThm5.2}) and hence all the conclusions above hold.
From the above  Euler adjoint inclusion
%\textcolor{red}{ and transversality condition},
 we get $p_y(t)\equiv 0$. In the case where $\phi, h$ are strictly differentiable, we also get $p_x(t)=\lambda_\phi(t)$ a.e.. Hence by taking $p(t)=p_x(t)$, the conclusions follow.
\end{proof}

We now compare Theorem \ref{thm5.1} (treating $y$ as a control variable)  with Theorem \ref{thm5.2} (treating $y$ as a state variable).
It is obvious that the WBCQ in (\ref{WBCQThm5.2}) implies
 (\ref{WBCQw}) and so  the WBCQ required for treating $y$ as a control variable is {weaker}.
In the case where $\phi, h$ are strictly differentiable, all conclusions except the Weierstrass condition are the same. The Weierstrass condition for treating $y$ as control is stronger since it implies the one for treating $y$ as a state variable. In summary, treating $y$ as control gives stronger necessary optimality conditions under {weaker} constraint qualifications. But this is not surprising since
treating $y$ as state variables requiring $y$ to be absolutely continuous while treating $y$ as control only requires $y$ to be weaker, i.e., only measurable.

\section{Discussion of constraint qualifications}
In this session  we discuss  sufficient conditions for constraint qualifications required in Theorems \ref{thm4.1} and \ref{thm5.1} to hold. The sufficient conditions for constraint qualifications required in  other necessary optimality conditions are similar.

We first discuss   sufficient conditions for constraint qualifications for Theorem \ref{thm4.1} to hold. The constraint qualifications involve the WBCQ plus the calmness of  the set-valued map $M_\varphi$ defined as in (\ref{perturb4.1n}).
%The sufficient condition for calmness of  $M_\varphi$ defined as in (\ref{perturb4.1n}) is similar and hence omitted.

It is easy to check that the calmness condition of $M_\varphi$ at $(0, \bar x, \bar u, \bar v)$ holds if
and only if the system defining the set $M_\varphi(0)$ has a local error bound at $(\bar x, \bar u,\bar v)$ (see e.g. \cite{hen}).
There are many sufficient conditions under which the local error bound holds (see e.g. Wu and
Ye \cite{Wu-Ye01,Wu-Yemp,Wu-Ye03}). However {not} many of them are easy to verify. Two easiest  criteria for checking   the calmness of the set-valued map $M_\varphi$ {are} the linear CQ and and NNAMCQ as defined in Proposition \ref{Prop6.2}(i)(iv) respectively.
Although the linear CQ and NNAMCQ are easy to verify, they may be still too strong for some problems to hold. Recently some new constraint qualifications that are stronger than calmness and
weaker than the linear CQ and/or NNAMCQ for  nonlinear
programs
%with equality and inequality constraints
have been introduced in the literature (see e.g. {\cite{abdreabi-haeser--,abdeabi--twoCQ,guoyezhang,helmut,helmutYe}).
%These CQs include the relaxed
%constant positive linear dependence condition (RCPLD) (see \cite{abdreabi-haeser--}), the
%constant rank of the subspace component condition (see \cite[Theorem 4.2]{abdeabi--twoCQ}) and
%the quasinormality \cite[Theorem 5.2]{guoyezhang}.
%Note that the calmness of the set-valued map $M $ is equivalent to the metric subregularity of the set-valued map $\Sigma(x,u):=-(\Phi(x,u), u )+ \Omega\times U$. Recently by using the directional normal cones, some new sufficient conditions for metric subregularity  have been developed (see \cite{helmut,helmutYe}).
%
%
For convenience, we summarize some prominent  verifiable sufficient conditions for the WBCQ plus the calmness of $M_\varphi$ as follows.
\begin{proposition}\label{Prop6.2}   Let $(\bar x,\bar{u},\bar{v})\in M_\varphi(0)$, $\varphi$ is Lipschitz continuous at $(\bar x,\bar u,\bar v)$ and $U, K_\varphi$ are closed.
Then the WBCQ
\begin{equation}
\begin{array}{l} {\lambda\in N_{K_\varphi}(\varphi(\bar x,\bar u,\bar v)),} (\alpha,0,0)\in \partial\langle \lambda, \varphi \rangle(\bar x,\bar u,\bar v)+\{0\}\times N_{U}(\bar u)\times \{0\}
\end{array} \Longrightarrow \alpha=0 \label{WBCQ1newn}
\end{equation}  and  the set-valued map $M_\varphi$ defined as in (\ref{perturb4.1n}) is calm at $(0,\bar x,\bar{u},\bar v)$ if one of
the following conditions holds:
\begin{itemize}
\item[\rm (i)] The WBCQ (\ref{WBCQ1newn}) and the  linear constraint qualification (Linear CQ) holds: $\varphi$ is affine and $U, K_\varphi$ are the union of finitely many polyhedral sets.
\item[\rm (ii)]  The CCQ  holds at $(\bar x, \bar u,\bar v)$: there exists $\mu >0$ such that
\begin{eqnarray*}
&& \lambda\in N_{K_\varphi}(\varphi(\bar x,\bar u,\bar v)), (\alpha,\beta,\gamma)\in \partial \langle \lambda, \varphi\rangle {(\bar{x},\bar{u}, \bar{v})}+\{0\}\times N_{U} (\bar{u})\times \{0\}\\
&& \qquad \qquad  \Longrightarrow |\lambda|\leq \mu |(\beta,\gamma)|.
\end{eqnarray*}
\item[\rm (iii)]  The MFC  holds at $(\bar x, \bar u,\bar v)$: \begin{eqnarray*}
\lambda\in N_{K_\varphi}(\varphi(\bar x,\bar u,\bar v)), (\alpha,0,0)\in \partial \langle \lambda, \varphi\rangle(\bar{x},\bar{u}, \bar{v})+\{0\}\times N_{U} (\bar{u})\times \{0\}
 \Longrightarrow \lambda=0.
\end{eqnarray*}
\item[\rm (iv)]    The NNAMCQ   holds at $(\bar x, \bar u,\bar v)$:
\begin{eqnarray*}
\lambda\in N_{K_\varphi}(\varphi(\bar x,\bar u,\bar v)), (0,0,0)\in \partial \langle \lambda, \varphi\rangle(\bar{x},\bar{u}, \bar{v})+\{0\}\times N_{U} (\bar{u})\times \{0\}
 \Longrightarrow \lambda=0.
\end{eqnarray*}
\item[\rm (v)] The WBCQ (\ref{WBCQ1newn}) and the  Quasinormality holds at $(\bar{x},\bar{u},\bar v)$:
\begin{eqnarray*}
&& \left \{ \begin{array}{l}
(0,0,0)\in \partial\langle \lambda, \varphi\rangle(\bar{x},\bar{u},\bar v)+\{0\}\times N_{U}(\bar{u})\times \{0\},\,\,
\lambda\in N_{K_\varphi}(\varphi(\bar{x},\bar{u},\bar v)),\\
\exists  (x^k,u^k,v^k,y^k,\lambda^k)\xrightarrow{\mathbb{R}^{n_{x}}\times U \times \mathbb{R}^{n_{x}} \times K_\varphi\times \mathbb{R}^{m}} (\bar{x},\bar{u},\bar v,\varphi(\bar{x},\bar{u},\bar v),\lambda)\\
 \mbox{ such that for each }  k,  \lambda_i\neq 0 \Longrightarrow  \lambda_i (\varphi_i(x^k,u^k,v^k)-y_i^k) >0
\end{array} \right \}
\Longrightarrow \lambda=0.
\end{eqnarray*}
\item[\rm (vi)]  The WBCQ (\ref{WBCQ1newn}) and the first order sufficient condition for metric subregularity (FOSCMS) at $(\bar{x},\bar{u},\bar v)$: $\varphi$ is   differentiable at $(\bar{x},\bar{u},\bar v)$, and for every $0\neq d:=(d_1,d_2,d_3)\in \mathbb{R}^{n_{x}}\times \mathbb{R}^{n_{u}}\times \mathbb{R}^{n_{x}}$ with $\nabla\varphi(\bar{x},\bar{u},\bar v)d\in T_{K_\varphi}(\varphi(\bar{x},\bar{u},\bar v)), d_2\in  T_{U}(\bar{u})$ one has
    \begin{eqnarray*}
\left \{ \begin{array}{l}
{(0,0,0)}\in \nabla\varphi(\bar{x},\bar{u},\bar v)^{T}\lambda+\{0\}\times {N_{U}^{L}(\bar{u};d_2)}\times \{0\},\\
\lambda\in N_{K_\varphi}(\varphi(\bar{x},\bar{u},\bar v);\nabla\varphi(\bar{x},\bar{u},\bar v)d)
\end{array} \right \}  \hspace{-0.3cm}&& \Longrightarrow\lambda=0.
\end{eqnarray*}
\item[\rm (vii)] The WBCQ (\ref{WBCQ1newn}) and the second order sufficient condition for metric subregularity (SOSCMS) at $(\bar{x},\bar{u},\bar v)$: $\varphi$ is twice
 Fr\'{e}chet differentiable at $(\bar{x},\bar{u},\bar v)$ and $K_\varphi, U$ are the union of finitely many convex polyhedra sets, and for every $0\neq d:=(d_1,d_2,d_3)\in \mathbb{R}^{n_{x}}\times \mathbb{R}^{n_{u}}\times \mathbb{R}^{n_{x}}$ with $\nabla\varphi(\bar{x},\bar{u},\bar v)d\in T_{K_\varphi}(\varphi(\bar{x},\bar{u},\bar v)), d_2\in  T_{U}(\bar{u})$ one has
    \begin{eqnarray*}
\left \{ \begin{array}{l}
{(0,0,0)}\in \nabla\varphi(\bar{x},\bar{u},\bar v)^{T}\lambda+\{0\}\times {N_{U}^{L}(\bar{u};d_2)}\times {\{0\}},\,\\
\lambda\in N_{K_\varphi}(\varphi(\bar{x},\bar{u},\bar v);\nabla\varphi(\bar{x},\bar{u},\bar v)d),\\
d^{T}\nabla^{2}\langle\lambda,\varphi\rangle(\bar{x},\bar{u},\bar v)d\geq 0 \
\end{array} \right\}  \hspace{-0.6cm}&& \Longrightarrow\lambda=0.
\end{eqnarray*}
\item[\rm (viii)] The WBCQ (\ref{WBCQ1newn}) and {the relaxed constant positive linear dependence (RCPLD) } holds at $(\bar{x},\bar{u},\bar{v})$: $\varphi$
%:=(g,h)$ with $g:\mathbb{R}^{n_x+n_u+n_x}\rightarrow \mathbb{R}^{m_1}, h:\mathbb{R}^{n_x+n_u+n_x}\rightarrow \mathbb{R}^{m_2}$
 is differentiable at $(\bar{x},\bar{u},\bar{v})$, $U=\mathbb{R}^{n_{u}}$, $K_\varphi=\mathbb{R}^{m_1}\times \mathbb{R}^{m-m_1}_+$,
${\cal J}\subseteq\{1,\cdots, m_1\}$
is such that $\{\nabla \varphi_j(\bar{x},\bar{u},\bar{v})\}_{j\in{\cal J}}$ is a basis for the $
span\{\nabla \varphi_j(\bar{x},\bar{u},\bar{v})\}_{j=1}^{m_1}$ and
 there exists
$\delta>0$ such that
\begin{itemize}
\item $\{\nabla \varphi_j({x},{u},{v})\}_{j=1}^{m_1}$ has the same rank for each $(x,u,v)\in B((\bar x, \bar u,\bar v),\delta)$;
\item For every ${\cal I}\subseteq I(\bar x,\bar u,\bar v):=\{i\in \{m_1+1,\dots, m\}: \varphi_i(\bar x,\bar u,\bar v)=0\}$,  if there exists $\{\lambda_j\}_{ {\cal J\cup I}}$ with $j\geq 0 \ \, \forall j\in {\cal I}$ not all zero such that
 \begin{eqnarray*}
\sum\limits_{j\in{\cal J\cup  I}}\lambda_{j}\nabla \varphi_{j}(\bar{x},\bar u, \bar{v})=0,
%\lambda\in N_{\Omega}(\Phi(\bar{x},\bar{u});\nabla h(\bar{x},\bar{u})d)
 %\Longrightarrow\lambda=0.
\end{eqnarray*}
then  $\{\nabla \varphi_j(x,u,v)\}_{ j\in {\cal J \cup I}}$ is linearly dependent
for each $(x,u,v)\in {B}((\bar{x},\bar{u},\bar{v}),\delta)$.
\end{itemize}
\end{itemize}

\end{proposition}
\begin{proof} (i) Under Linear CQ, the set-valued map $M_\varphi$ is a polyhedral multifunction and hence upper Lipschitz continuous as shown by Robinson \cite{robin1}. The results follows from the fact that the upper Lipschitz continuity implies the calmness.

(ii)-(v)  By definition, it is easy to see that
$$\mbox{CCQ} \Rightarrow \mbox{MFC} \Rightarrow \mbox{NNAMCQ} \Rightarrow \mbox{WBCQ} \mbox{ and }  \mbox{NNAMCQ}  \Longrightarrow \mbox{Quasinormality}.$$
%$$ \mbox{CCQ} \Longrightarrow  \mbox{MFC} \Longrightarrow \mbox{NNAMCQ}  \Longrightarrow \mbox{Quasinormality}.$$
By  \cite[Theorem 5.2]{guoyezhang}, the quasinomality implies the calmness.

(vi)  Let $q(x,u):=(\varphi(x,u,v),u)\in \Gamma:=\Omega\times U$.  Note that the calmness of the set-valued map $M_\varphi(\cdot)$ at $(0, \bar x, \bar u,\bar v)$ is equivalent to the metric subregularity of the set-valued map $\Sigma(x,u,v):=q(x,u,v) -\Gamma$ at $(\bar x,\bar u, \bar v,0)$. By  \cite[1. of Corollary 1]{helmut2}, it suffices to show that for every $0\not =w  $ with $\nabla q(\bar{x},\bar{u},\bar v)w\in T_\Gamma (q(\bar{x},\bar{u},\bar v))$ one has
$$\nabla q(\bar{x},\bar{u},\bar v)^T \eta=0, \eta \in N_{\Gamma}(q(\bar{x},\bar{u},\bar v);\nabla q(\bar{x},\bar{u},\bar v)w) \Longrightarrow \eta=0.$$
%For every $0\neq d=(d_{1},d_{2})\in\mathbb{R}^{n_{x}}\times\mathbb{R}^{n_{u}}$ with $\nabla\Phi(\bar{x},\bar{u})d\in T_{\Omega}(\Phi(\bar{x},\bar{u}))$ and $ \left(
%                                                 \begin{array}{cc}
%                                                   0 & 0\\
%                                                   0 & I \\
%                                                 \end{array}
%                                               \right)
%d=d_{2}\in T_{U}(\bar{u})$
%such that
%$$ \nabla\Phi(\bar{x},\bar{u})^{T}\lambda+\left(
%                                                 \begin{array}{cc}
%                                                   0 & 0\\
%                                                   0 & I \\
%                                                 \end{array}
%                                               \right)^{T}\mu=0,\,\, \lambda\in  N_{\Omega}(\Phi(\bar{x},\bar{u});\nabla\Phi(\bar{x},\bar{u})d), \mu\in N_{U}(\bar{u};d_{2}).
%$$
By \cite[Proposition 3.3]{YeZhou}, we have
\begin{eqnarray*}
&& T_\Gamma(q(\bar x,\bar u, \bar v)) \subseteq T_{K_\varphi} (\Phi(\bar x,\bar u,\bar v)) \times T_U(\bar u),\\
&& N_\Gamma (q(\bar x,\bar u,\bar v);\nabla q(\bar{x},\bar{u},\bar v)u ) \subseteq N_\Omega(\varphi(\bar x,\bar u,\bar v); \nabla \varphi(\bar{x},\bar{u},\bar v)u) \times N_U(\bar u; d_2),
\end{eqnarray*}
and the equality holds if at most one of the sets $K_\varphi, U $ is directionally regular.
%By definition, It is easy to see that $ N_{U}(\bar{u};d_{2})\subseteq N_{U}(\bar{u})$.
%Hence easy calculation shows that
Hence the FOSCMS defined as in (vi)  is stronger than the condition required above and  the calmness holds.
%$$
%(0,0)\in \nabla\Phi(\bar{x},\bar{u})^{T}\lambda+\{0\}\times N_{U}^{L}(\bar{u}),\,\, \lambda\in  N_{\Omega}(\Phi(\bar{x},\bar{u});\nabla\Phi(\bar{x},\bar{u})d), \mu\in N_{U}^{L}(\bar{u}).
%$$
%Then $\lambda=0$ by the assumption. Hence  we get that $M$ is calm at $(0,\bar{x},\bar{u})$ by \cite[Corollary 1]{helmut2}.\\

(vii)  By the same arguments as above, we can verify that the SOSCMS  satisfies the condition of \cite[2. of Corollary 1]{helmut2}. So the result holds.

(viii)  follows from \cite[Theorem 4.2]{guo-zhang}.
\end{proof}

Now we discuss sufficient conditions for  constraint qualifications for Theorem \ref{thm5.1}  to hold. The constraint qualifications involve the WBCQ plus the calmness of  the set-valued map $M_h$ defined as in (\ref{Mh}) where we treat $y$ as a control. The proof of the results  are similar to Proposition \ref{Prop6.2} and hence we omit it.
\begin{proposition}\label{Prop6.3}   Let $(\bar x,\bar{y},\bar{u})\in M_h(0)$, $h$ is Lipschitz continuous at $(\bar x, \bar y, \bar u)$ and $U$ is closed.
Then  the WBCQ
\begin{equation}  \lambda\in \mathbb{R}^{n_y},  (\alpha,0,0)\in \partial \langle \lambda, h\rangle(\bar x,\bar y,\bar u)+\{(0,0)\}\times N_U (\bar u)
\Longrightarrow \alpha=0\label{WBCQ2}
\end{equation} and
the set-valued mapping $M_h$  defined as in (\ref{Mh}) is calm at $(0,\bar{x},\bar{y}, \bar{u})$ if one of
the following conditions holds:
\begin{itemize}
\item[\rm (i)] The  WBCQ (\ref{WBCQ2}) and  the linear constraint qualification (Linear CQ) holds: $h$ is affine and $U$ is  the union of finitely many polyhedral sets.
\item[\rm (ii)]  The CCQ holds at $(\bar x, \bar y, \bar u)$:  there exists $\mu >0$ such that
\begin{eqnarray*}
\lambda\in \mathbb{R}^{n_y}, (\alpha,\beta,\gamma)\in \partial \langle \lambda, h\rangle(\bar{x},\bar{y}, \bar{u})+\{(0,0)\}\times N_{U} (\bar{u})
 \Longrightarrow |\lambda|\leq \mu |(\beta,\gamma)|.
\end{eqnarray*}
\item[\rm (iii)]  The MFC holds at $(\bar x, \bar y,\bar u)$:
\begin{eqnarray*}
\lambda\in \mathbb{R}^{n_y}, (\alpha,0,0)\in \partial \langle \lambda, h\rangle(\bar{x},\bar{y}, \bar{u})+\{(0,0)\}\times N_{U}(\bar{u})
 \Longrightarrow \lambda=0.
\end{eqnarray*}
\item[\rm (iv)]    The NNAMCQ  holds at $(\bar x, \bar y, \bar u)$:
\begin{eqnarray*}
\lambda\in \mathbb{R}^{n_y}, (0,0,0)\in \partial \langle \lambda, h\rangle(\bar{x},\bar{y}, \bar{u})+\{(0,0)\}\times N_{U}(\bar{u})
 \Longrightarrow \lambda=0.
\end{eqnarray*}
\item[\rm (v)] The  WBCQ (\ref{WBCQ2}) and  the quasinormality holds at $(\bar{x},\bar y, \bar{u})$:
\begin{eqnarray*}
&& \left \{ \begin{array}{l}
(0,0,0)\in \partial\langle \lambda, h\rangle(\bar{x}, \bar y, \bar{u})+\{(0,0)\}\times N_{U}(\bar{u}),\\
\exists  (x^k,y^k, u^k,\lambda^k)\xrightarrow{\mathbb{R}^{n_{x}}\times  \mathbb{R}^{n_{y}} \times U \times \mathbb{R}^{d}} (\bar{x}, \bar y, \bar{u}, \lambda)\\
 \mbox{ such that for each }  k,  \lambda_i\neq 0 \Longrightarrow  \lambda_i  h_i(x^k,y^k, u^k) >0
\end{array} \right \}
\Longrightarrow \lambda=0.
\end{eqnarray*}
\item[\rm (vi)] The WBCQ (\ref{WBCQ2}) and  the FOSCMS at $(\bar{x},\bar y, \bar{u})$: $h$ is   differentiable at $(\bar{x},\bar y, \bar{u})$, and for every $0\neq d:=(d_1,d_2)\in \mathbb{R}^{n_{x}+n_y}\times \mathbb{R}^{n_{u}}$ with $\nabla h(\bar{x},\bar y, \bar{u})d=0, d_2\in  T_{U}(\bar{u})$ one has
    \begin{eqnarray*}
\lambda\in \mathbb{R}^{n_y}, (0,0,0)\in \nabla h(\bar{x},\bar y, \bar{u})^{T}\lambda+\{(0,0)\}\times {N_{U}^{L}(\bar{u};d_2)}
%\lambda\in N_{\Omega}(\Phi(\bar{x},\bar{u});\nabla h(\bar{x},\bar{u})d)
 \Longrightarrow\lambda=0.
\end{eqnarray*}
\item[\rm (vii)] The WBCQ (\ref{WBCQ2}) and  SOSCMS at $(\bar{x},\bar y,\bar{u})$: $h$ is twice
 Fr\'{e}chet differentiable at $(\bar{x},\bar{u})$, $ U$ is the union of finitely many convex polyhedra sets, and for every $0\neq d:=(d_1,d_2)\in \mathbb{R}^{n_{x}+n_y}\times \mathbb{R}^{n_{u}}$ with
 $\nabla h(\bar{x},\bar y, \bar{u})d=0, d_2\in  T_{U}(\bar{u})$ one has
    \begin{eqnarray*}
\left \{ \begin{array}{l}
(0,0,0)\in \nabla h(\bar{x}, \bar{y}, \bar{u})^{T}\lambda+\{(0,0)\}\times {N_{U}^{L}(\bar{u};d_2)},\,\\
d^{T}\nabla^{2}\langle\lambda,h\rangle(\bar{x},\bar{y},\bar{u})d\geq0,\lambda\in \mathbb{R}^{n_y} \
\end{array} \right\}  \hspace{-0.6cm}&& \Longrightarrow\lambda=0.
\end{eqnarray*}
\item[\rm (viii)] The WBCQ (\ref{WBCQ2}) and {the constant rank constraint qualification  (CRCQ)} at $(\bar{x},\bar{y},\bar{u})$: suppose $h$ is differentiable around $(\bar{x},\bar{y},\bar{u})$ and  $U=\mathbb{R}^{n_{u}}$,
%${\cal J}\subseteq\{1,\cdots, n_y\}$
%be such that $\{\nabla h_j(\bar{x},\bar{y},\bar{u})\}_{j\in{\cal J}}$ is a basis for the ${\rm
%span}\,\{\nabla h_j(\bar{x},\bar{y},\bar{u})\}_{j=1}^{n_y}$ and
 there exists
$\delta>0$ such that
%\begin{itemize}
%\item
 $\{\nabla h_j(x,y,u)\}_{j=1}^{n_y}$ has the same rank for each $(x,y,u)\in {
B}((\bar{x},\bar{y},\bar{u}),\delta)$.
%\item  if there exists $\{ \lambda_j\in \mathbb{R}^{n_y}\, | \ j\in{\cal J}\}$ not all zero such that
% \begin{eqnarray*}
%(0,0,0)\in \sum\limits_{j\in{\cal J}}\lambda_{j}\nabla h_{j}(\bar{x},\bar y, \bar{u})+\{(0,0)\}\times N_{U}(\bar{u}),
%%\lambda\in N_{\Omega}(\Phi(\bar{x},\bar{u});\nabla h(\bar{x},\bar{u})d)
% %\Longrightarrow\lambda=0.
%\end{eqnarray*}

%then  $\{\nabla h_j(x,y,u)\}_{ j\in {\cal J}}$ is linearly dependent
%for each $(x,y,u)\in{B}((\bar{x},\bar{y},\bar{u}),\delta)$.
%\end{itemize}
\end{itemize}

\end{proposition}
%\begin{proof} (i)-(vii) The results follows directly from Proposition \ref{Prop6.1} by the fact that we view $(y,u)$ here as the control variable $u$  and choose $\Phi:=h,\,\Omega:=\{0\}$ in Proposition \ref{Prop6.1}.
%
% (viii)  follows from \cite[Theorem 4.2]{guo-zhang}.
%\end{proof}

%It is worth noting that  in \cite{pinho*}, a special class of control of DAEs where  $\varphi (x, u,\dot{x}):= E\dot{x} -g(x, u)$ with $E$ an $ m\times n_x$ matrix and $g:\mathbb{R}^{n_x}\times \mathbb{R}^{n_u}\rightarrow \mathbb{R}^m$  a smooth mapping  is studied as a special case of the result for  the  implicit control system. Depending on the rank of the matrix $E$, the following three
%cases are considered.
%
%{\bf Case (A)} $E$ is of full row rank;
%
%{\bf Case (B)} $E$ is of full column rank;
%
%{\bf Case (C)} $E$ is of neither of full row rank nor of column rank.

To compare with  \cite[Section 4]{pinho*},  next we consider a special case of  $(P_{DAE})$ with  $\varphi(x,u,v):=Ev-g(x,u)$ and $K_\varphi=\{0\}$
\begin{eqnarray*}
(P^{'}_{DAE})~~~\min &&  f(x(t_0),x(t_1))\nonumber \\
s.t. && E\dot{x}(t)-g(x(t), u(t))=0,\\
&& u(t) \in U \quad a.e. \,t\in [t_0,t_1],\nonumber\\
&& (x(t_0), x(t_1))\in S,\nonumber
\end{eqnarray*}
where $E$ is a $m\times n_{x}$ matrix with $rank(E)=r$, $g:\mathbb{R}^{n_{x}}\times \mathbb{R}^{n_{u}}\rightarrow \mathbb{R}^{m}$ is strictly differentiable.
 Depending on the rank of the matrix $E$, the following three
cases are considered in  \cite[Section 4]{pinho*}

{\bf Case (A)} $E$ is of full row rank;

{\bf Case (B)} $E$ is of full column rank;

{\bf Case (C)} $E$ is of neither of full row rank nor of column rank.

Note that  \cite{pinho*} allows for the dynamic system to be nonautonomous  but the matrix $E$ is required to have  some special forms. For those special matrix $E$, depending on the cases, de Pinho  \cite{pinho*} augmented the system and transform the original problem to the one that may be easier to analyze.
%If matrix $E$ is of full row rank ($m< n_{x}$ and $r=m$), w

In case (A),  we obtain the following results as a corollary of Corollary \ref{thm5.1new}.
\begin{corollary} \label{thm6.1new} Let  $(x_*,u_*)$ be a  local minimum of radius $R(\cdot)$ for $(P^{'}_{DAE})$.
  Suppose that
  $E$ is of full row rank and that
   there exists $\delta>0$ such that $R(t)\geq \delta$.
 Suppose further that
 ${C}_\varphi^{\varepsilon,R}$ as defined in (\ref{C}) with $K_\varphi=\{0\}$  is compact.
% for  all $(t,x,u,v) \in {C}_\varphi^{\varepsilon, R}$  the WBCQ holds:
%\begin{eqnarray*}
%\begin{array}{l}  (\alpha,0,0)\in \partial \langle \lambda_{\varphi}, \varphi\rangle(x,u,v)+\{0\}\times N_U (u)\times \{0\}, \quad
%\lambda_{\varphi}\in \mathbb{R}^k
%%\lambda_{\psi}\in N_{K_{\psi}}^{L}(\psi(x,u)), \beta\in N_{U}^{L}(u)
%\end{array} \Longrightarrow \alpha=0
%\end{eqnarray*}
%and the mapping ${M}_\varphi$ as defined in (\ref{perturb4.1n}) with $K_\varphi=\{0\}$ is calm at $(0, x,u,v)$.
Then there exist an arc $p$, a number $\lambda _{0}$ in $\{0,1\}$
%and  a measurable function $\lambda:[t_0,t_1]\rightarrow \mathbb{R}^n$ with
% for almost every $t\in [t_0,t_1]$
and a measurable function
%s $\lambda_{\varphi}:[t_0,t_1]\rightarrow \mathbb{R}^m$,
$\mu:[t_0,t_1]\rightarrow \mathbb{R}^{n_{u}}$ with $\mu(t)\in N_{U}^{C}(u_*(t))$ a.e.
satisfying the nontriviality condition
$(\lambda _{0},p(t))\neq0, \forall t\in[t_0,t_1]$,
 the transversality condition
$$(p(t_0),-p(t_1)) \in \lambda_0 \partial f(x_*(t_0),x_*(t_1))+N_S (x_*(t_0),x_*(t_1)),$$
and the Euler adjoint inclusion for almost every $t$:
\begin{eqnarray}\label{6.3}
\left \{ \begin{array}{l}
\dot{p}(t)=-\nabla_{x}g(x_*(t),u_*(t))^{T} (EE^{T})^{-1}Ep(t), \,\mbox {a.e.},\\
\mu(t)=\nabla_{u}g(x_*(t),u_*(t))^{T} (EE^{T})^{-1}Ep(t),  \,\mbox {a.e.},
%\lambda_{\varphi}(t)=(EE^{T})^{-1}Ep(t) \
\end{array} \right.
\end{eqnarray}
as well as the Weierstrass condition of radius $R(\cdot)$ for almost every $t$:\begin{eqnarray*}
{ u\in U, Ev=g(x_*(t), u), |(u,v)-(u_*(t),\dot{x}_*(t))| < R(t)\Longrightarrow }
 \langle p(t),v-\dot{x}_*(t) \rangle\leq 0.
\end{eqnarray*}
%Suppose further that $\varphi$ is {strictly} differentiable, then  the Euler adjoint inclusion can be expressed in the explicit form: there exists
%measurable functions $\lambda_{\varphi}:[t_0,t_1]\rightarrow \mathbb{R}^m$,
%%$\lambda_{\psi}:[t_0,t_1]\rightarrow \mathbb{R}^d$,
%$\mu:[t_0,t_1]\rightarrow \mathbb{R}^{n_{u}}$ with
%%$\lambda_{\varphi}(t)\in N_{K_{\varphi}}^{C}(\varphi(x_*(t),u_*(t),\dot{x}_*(t)))$,
%%$\lambda_{\psi}(t)\in N_{K_{\psi}}^{C}(\psi(x_*(t),u_*(t)))$,
%$\mu(t)\in N_{U}^{C}(u_*(t))$ a.e. such that
%$$(\dot{p}(t),-\mu(t), p(t)) \in
% \lambda_0\partial^CF( x_*(t),u_*(t),\dot{x}_*(t))+\nabla \varphi(x_*(t),u_*(t),\dot{x}_*(t))^T\lambda_{\varphi}(t)
%%+ \nabla \psi(x_*(t),u_*(t))^T\lambda_{\psi}(t),
%\mbox { a.e. }.$$
%If $N_U^C(u_*(t))=\{0\}$, then the estimate for the  multiplier $\lambda_{\varphi}(t)$  also holds:
%$$
%|\lambda_{\varphi}(t)|\leq k|p(t)| \qquad a.e.
%$$
%for some positive constant $k>0$.
In the case of free end point, $\lambda_0$ can be taken as $1$.
\end{corollary}
\begin{proof}
Since $\nabla_v\varphi =E$ is of full row rank,
%By the fact that  E is of full row rank and g is strictly differentiable, it is easy to verify that
%\begin{eqnarray*}
%\lambda_{\varphi}\in \mathbb{R}^m,\,\, (0,0,0)\in \nabla \langle \lambda_{\varphi}, \varphi\rangle(x,u,v)+\{0\}\times N_U (u)\times \{0\}
%%\lambda_{\psi}\in N_{K_{\psi}}^{L}(\psi(x,u)), \beta\in N_{U}^{L}(u)
% \Longrightarrow \lambda_{\varphi}=0,
%\end{eqnarray*}
 the NNAMCQ holds automatically at any feasible point. By Proposition \ref{Prop6.2}(iv), WBCQ plus the calmness of $M_{\varphi}$  holds.
Hence all the assumptions in Corollary \ref{thm5.1new} are satisfied.
By Corollary \ref{thm5.1new}, there exist an arc $p$, a number $\lambda_0\in \{0,1\}$, and measurable functions
$\lambda_{\varphi}:[t_0,t_1]\rightarrow \mathbb{R}^m$,
$\mu:[t_0,t_1]\rightarrow \mathbb{R}^{n_{u}}$ satisfying the nontriviality condition,  the transversality condition, the Euler adjoint inclusion  and the Weierstrass condition. We only need to prove the Euler adjoint inclusion (\ref{6.3}).
By the Euler adjoint inclusion in Corollary \ref{thm5.1new}, we have
$$p(t)=\nabla_v\varphi^T \lambda_\varphi=E^T \lambda_\varphi.$$  Since $E$ is of full row rank, we can solve $\lambda_\varphi= (EE^T)^{-1}Ep(t)$ from the above linear system and hence the proof is completed.
\end{proof}

If $E=\left(
        \begin{array}{cc}
          E_{a} & 0 \\
        \end{array}
      \right)
$ where $E_{a}$ is a $m\times m$ nonsingular matrix, the results obtained for the case (A) are the same as that of \cite[Corollary 4.1]{pinho*} but without requiring the restriction for the function $f$.
% But we don't need extra constraint qualification, and we can get the estimate of $\lambda_{\varphi}(t)$.
%Compared with  \cite[Corollary 4.1]{pinho*},  the necessary optimality conditions coincides  with that of \cite[Corollary 4.1]{pinho*} when $R(t)\equiv\infty$. The differences in the Euler adjoint inclusion and Weierstrass condition lie in that we do not  need to partition $E$.

In case (B) and (C),  we obtain the following results as a corollary of Corollary \ref{thm5.1new}.
\begin{corollary} \label{thm6.2new} Let  $(x_*,u_*)$ be a  local minimum of radius $R(\cdot)$ for $(P_{DAE})$.
  Suppose that $E$ is not of full row rank but one of assumptions in Proposition \ref{Prop6.2}(i)(v)(vi)(vii)(viii) holds. Suppose further that
 ${C}_\varphi^{\varepsilon,R}$ as defined in (\ref{C}) with $K_\varphi=\{0\}$  is compact  and  there exists $\delta>0$ such that $R(t)\geq \delta$. Then there exist an arc $p$, a number $\lambda _{0}$ in $\{0,1\}$
%and  a measurable function $\lambda:[t_0,t_1]\rightarrow \mathbb{R}^n$ with
% for almost every $t\in [t_0,t_1]$
and  measurable functions $\lambda_{\varphi}:[t_0,t_1]\rightarrow \mathbb{R}^m$,
$\mu:[t_0,t_1]\rightarrow \mathbb{R}^{n_{u}}$ with $\mu(t)\in N_{U}^{C}(u_*(t))$ a.e.
satisfying the nontriviality condition,
 the transversality condition, the Weierstrass condition as in Corollary \ref{thm6.1new}
and the Euler adjoint inclusion for almost every $t$:
% Suppose further that
% ${C}_\varphi^{\varepsilon,R}$ as defined in (\ref{C}) with $K_\varphi=\{0\}$  is compact and
% for  all $(t,x,u,v) \in {C}_\varphi^{\varepsilon, R}$  the WBCQ holds:
%\begin{eqnarray*}
%  \lambda_{\varphi}\in \mathbb{R}^m, \,\,(\alpha,0,0)\in \nabla \varphi(x,u,v)^T \lambda_{\varphi}+\{0\}\times N_U (u)\times \{0\}
% \Longrightarrow \alpha=0.
%\end{eqnarray*}
%and the mapping ${M}_\varphi$ as defined in (\ref{perturb4.1n}) with $K_\varphi=\{0\}$ is calm at $(0, x,u,v)$.
%\begin{eqnarray}\label{6.4}
%(\dot{p}(t),-\mu(t), p(t)) \in\nabla \varphi(x_*(t),u_*(t),\dot{x}_*(t))^T\lambda_{\varphi}(t)
%\mbox { a.e.}.
%\end{eqnarray}
\begin{eqnarray}\label{6.5}
\left \{ \begin{array}{l}
\dot{p}(t)=-\nabla_{x}g(x_*(t),u_*(t))^{T} \lambda_{\varphi}(t), \,\mbox {a.e.},\\
\mu(t)=\nabla_{u}g(x_*(t),u_*(t))^{T} \lambda_{\varphi}(t),  \,\mbox {a.e.},\\
p(t)=E^{T}\lambda_{\varphi}(t).\
\end{array} \right.
\end{eqnarray}
If $N_U^C(u_*(t))=\{0\}$, then  the estimate for the  multiplier $\lambda_{\varphi}(t)$  also holds:
$$
|\lambda_{\varphi}(t)|\leq k|p(t)| \qquad a.e.
$$
for some positive constant $k>0$.
In the case of free end point, $\lambda_0$ can be taken as $1$.
\end{corollary}

\section*{Acknowledgments} We thank the anonymous reviewers of this paper for valuable  comments that helped us to improve the presentation of the manuscript.

\section*{Appendix: Proof of Theorem \ref{thm4.2new}}

Our proof is based on the following result.

%The local minimum  concept here is slightly different from  the one given in \cite{anli}   where  the  {\em $W^{1,1}$ local minimum of radius $R(\cdot)$} means that $(x_{*},u_{*})$  minimizes  $J(x, u)$ over all admissible pairs $(x,u)$ which satisfies both $| x(t)-x_{*}(t)|\leq \varepsilon$, $|\dot{x}(t)-\dot{x}_*(t)|\leq R(t) \mbox{ a.e.}$ and
%$\int_{t_0}^{t_1}|\dot{x}(t)-\dot{x}_*(t)|dt\leq \varepsilon$.
% This definition is very suitable for differential inclusions but might not be  so for  optimal control problems. Therefore we improve the result from \cite{anli}.
For convenience, we first recall the following result from \cite{anli}.
For any given $\varepsilon>0$ and a given radius function $R(t)$, define
$$S_*^{\varepsilon,R}(t):=\{(x,u)\in \bar{B}(x_*(t),\varepsilon)\times U:  \Phi(x,u) \in \Omega , |\phi(x,u)-\dot{x}_*(t)|\leq R(t)\},$$
$$C_*^{\varepsilon,R}=cl\{(t,x,v)\in [t_0,t_1]\times \mathbb{R}^{n_x}\times \mathbb{R}^{n_x}: v=\phi(x,u),(x,u)\in S_*^{\varepsilon,R}(t)\},$$
where $cl$ denotes the closure.

\begin{proposition}\cite[Theorem 4.2]{anli}\label{Propnew}
 Let  $(x_*,u_*)$ be a $W^{1,1}$ local minimum of radius $R(\cdot)$ for $(P)$  in the sense that
 that $(x_{*},u_{*})$  minimizes  $J(x, u)$ over all admissible pairs $(x,u)$ which satisfies both $| x(t)-x_{*}(t)|\leq \varepsilon$, $|\dot{x}(t)-\dot{x}_*(t)|\leq R(t) \mbox{ a.e.}$ and
$\int_{t_0}^{t_1}|\dot{x}(t)-\dot{x}_*(t)|dt\leq \varepsilon$.  Suppose {that}  there exists $\delta>0$ such that $R(t)\geq \delta$.
 Moreover suppose that $C_*^{\varepsilon, R}$ is compact and that for  all  $(t,x,u)$ with $(t,x, \phi(x,u))\in C_*^{\varepsilon,R}$, the WBCQ holds:
\begin{eqnarray*}
\left \{ \begin{array}{l}  (\alpha,0)\in \partial \langle \lambda, \Phi\rangle(x,u)+\{0\}\times N_{U}(u)\\
\lambda\in N_{\Omega}^{L}(\Phi(x,u))\end{array} \right. \Longrightarrow \alpha=0
\end{eqnarray*}
and the mapping $M$ defined as in (\ref{perturb3.1}) is calm at $(0, x,u)$.
%Suppose also that the tempered growth condition for radius $R$ near $x_*$ holds.
%If for some positive $r,c$, we have, for almost every $t$ and $(x,u)\in S_*^{\varepsilon,R}(t)$,
%$$ \bar{B}(0;r)\subseteq \Omega(t)-\{H(t,x,u)-\langle D_{u}H(t,x,u)+\zeta,u'-u\rangle:\zeta\in N_{U(t)}^L(u), |u'-u|\leq c\},$$
Then the transversality condition, the Euler adjoint inclusion in Theorem \ref{thm4.2new} hold and the Weierstrass condition of radius $R(\cdot)$ holds for almost every $t$: \begin{eqnarray*}
&&\Phi(x_*(t),u)\in \Omega, u\in U, \, \,  |\phi( x_*(t), u)-\phi( x_*(t),u_*(t))|< R(t)\Longrightarrow\\
&&\langle p(t),\phi(x_*(t),u) \rangle-\lambda_{0}F(x_*(t),u)\leq \langle p(t), \phi (x_*(t),u_*(t)) \rangle -\lambda_{0}F(x_*(t),u_*(t)).
\end{eqnarray*}
\end{proposition}

%Suppose also that the tempered growth condition for radius $R$ near $x_*$ holds.
%If for some positive $r,c$, we have, for almost every $t$ and $(x,u)\in S_*^{\varepsilon,R}(t)$,
%$$ \bar{B}(0;r)\subseteq \Omega(t)-\{H(t,x,u)-\langle D_{u}H(t,x,u)+\zeta,u'-u\rangle:\zeta\in N_{U(t)}^L(u), |u'-u|\leq c\},$$
%Then the necessary optimality conditions of Theorem \ref{thm4.1} hold as stated with the same radius $R$.

We now use Proposition \ref{Propnew} to prove Theorem \ref{thm4.2new}.

%{\bf Proof of Theorem \ref{thm4.2new} }%We shall first prove the theorem for the case  $F\equiv 0$.
% Let
%\begin{equation}
%\Gamma(x,y):=\{(\phi(x,u),\rho u): u\in U, \Phi(x,u) \in \Omega\}, \label{dynamic}
% \end{equation}
%where $\rho>1$ is a given parameter. The graph of $\Gamma$ is then defined as
%$$gph\Gamma:=\{(x,y,\phi(x,u),\rho u): u\in U, \Phi(x,u)\in \Omega\}.$$
Define $y_{*}(t)=\rho \int_{t_{0}}^{t}u_*(s)ds$  as well as a radius function $R_\rho(t):=\rho R(t)$ with $\rho>1$.
We claim that $(x_{*}, y_{*},u_*)$ is a $W^{1,1}$ local minimum  with radius $R_\rho(\cdot)$  for the following problem:
\begin{eqnarray*}
(P_\rho)~~~~~~\min &&  J(x,u):=\int_{t_0}^{t_1} F( x(t), u(t)) dt + f(x(t_0),x(t_1)),\nonumber \\
s.t. && \dot{x}(t)=\phi(x(t), u(t)) \,\quad a.e. \,t\in [t_0,t_1],\nonumber\\
&& \dot{y}(t)=\rho u(t)  \quad a.e. \,t\in [t_0,t_1], \\
&& \Phi(x(t), u(t))\in \Omega \quad a.e. \,t\in [t_0,t_1], \\
&& u(t) \in U \quad a.e. \,t\in [t_0,t_1],\nonumber\\
&& (x(t_0), x(t_1),y(t_0))\in S\times \{0\}.\nonumber
\end{eqnarray*}
%the differential inclusion optimal control problem $(P_{\Gamma})$
%\begin{eqnarray*}
%(P_{\Gamma})~~~~~\min &&  f(x(t_0),x(t_1))\\
%s.t. && (\dot{x}(t), \dot{y}(t))\in \Gamma(x(t),y(t)) \,\quad a.e. \,t\in [t_0,t_1]\\
%&& (x(t_0),x(t_1))\in S,  y(t_0))=0.
%\end{eqnarray*}
%over the arcs $(x,y)$ satisfying
%$$|(\dot{x}(t),\dot{y}(t))-(\dot{x}_*(t),\dot{y}_*(t))|\leq R_\Gamma (t), |x(t)-x_{*}(t)|\leq \varepsilon \ a.e., \  \int_{t_0}^{t_1}|\dot{x}(t)-\dot{x}_{*}(t)|dt\leq \varepsilon.$$
Let $(x,y,u)$ be  an admissible pair  for problem $(P_{\rho})$ satisfying
\begin{eqnarray}
&& |( \dot{x}(t),\dot{y}(t))-( \dot{x}_{*}(t),\dot{y}_{*}(t))|\leq R_\rho(t) \mbox{ a.e.},\label{adm1}\\
&& |(x(t), y(t))-(x_{*}(t), y_{*}(t))|\leq \varepsilon  \mbox{ a.e.},
 \int_{t_0}^{t_1}|(\dot{x}(t), \dot{y}(t))-(\dot{x}_{*}(t), \dot{y}_{*}(t))|dt\leq \varepsilon. \quad \label{adm2}
 \end{eqnarray}
Then it is obvious that  $(x(t),u(t))$ is an  admissible pair for $(P)$ with
 $$ |u(t)-u_{*}(t)|\leq R(t), |x(t)-x_{*}(t)|\leq \varepsilon  \mbox{ a.e.},
 \int_{t_0}^{t_1}|\dot{x}(t)-\dot{x}_{*}(t)|dt\leq \varepsilon. $$
 It follows  by the fact  that $(x_*,u_*)$ is a  local minimum of radius $R(\cdot)$ for $(P)$ that
  \begin{equation}
\int_{t_0}^{t_1} F( x_*(t), u_*(t)) dt + f(x_*(t_0),x_*(t_1)) \leq \int_{t_0}^{t_1} F( x(t), u(t)) dt +f(x(t_0),x(t_1)) .\label{ob}
 \end{equation}
% Furthermore, we know that $(x_*,u_*)$ is a $W^{1,1}$ local minimum of radius $R$ for $(P)$, and it follows that
Since (\ref{ob}) holding for all admissible pair $(x,y,u)$ satisfying (\ref{adm1})-(\ref{adm2}),  $(x_{*}, y_{*},u_*)$ is a $W^{1,1}$ local minimum of radius $R_{\rho}(\cdot)$ for $(P_{\rho})$.

%\textcolor{red}{
%As in the proof of  Proposition \ref{Propnew}, we define
%\begin{equation}
%\Gamma_{\rho}(x,y):=\{(\phi(x,u),\rho u): u\in U, \Phi(x,u) \in \Omega\}, \label{dynamic}
% \end{equation}
%where $\rho>1$ is a parameter,
%and construct the differential inclusion optimal control problem $(P_{\Gamma})$
%\begin{eqnarray*}
%(P_{\Gamma_{\rho}})~~~~~\min &&  f(x(t_0),x(t_1))\\
%s.t. && (\dot{x}(t), \dot{y}(t))\in \Gamma_{\rho}(x(t),y(t)) \,\quad a.e. \,t\in [t_0,t_1]\\
%&& (x(t_0),x(t_1))\in S,  y(t_0))=0.
%\end{eqnarray*}
%Then $(x_{*}, y_{*})$ is a $W^{1,1}$ local minimum  with radius $R_\rho(\cdot)$ for $(P_{\Gamma_{\rho}})$.
%%The graph of $\Gamma_{\rho}$ is then defined as
%%$$gph\Gamma_{\rho}:=\{(x,y,\phi(x,u),\rho u): u\in U, \Phi(x,u)\in \Omega\}.$$
%}
Denote by
\begin{eqnarray*}
\lefteqn{ S_*^{\varepsilon, R_\rho}(t)}\\
&& :=\left \{(x,y,u)\in \bar{B}((x_*(t),y_*(t)),\varepsilon)\times U: \begin{array}{l}
\Phi(x,u) \in \Omega,\\ |(\phi(x,u)-\dot{x}_*(t), \rho u-\rho u_*(t))|\leq \rho R(t)
\end{array}\right\},\\
\lefteqn{C_*^{\varepsilon, R_\rho}:=}\\
&&cl\{(t,x,y,\phi(x,u),\rho u)\in  [t_0,t_1]\times \mathbb{R}^{n_{x}}\times \mathbb{R}^{n_{u}} \times \mathbb{R}^{n_{x}}\times \mathbb{R}^{n_{u}}: (x,y,u) \in  S_*^{\varepsilon, R_\rho}(t)\}.
\end{eqnarray*}
It is obvious that the compactness of $\tilde{C}_*^{\varepsilon,R}$ implies the compactness of $C_*^{\varepsilon, R_\rho}$.
It is also obvious that $(t,x,y,\phi(x,u),\rho u)\in C_*^{\varepsilon, R_\rho}$ implies that $(t, x, u)\in \tilde{C}_*^{\varepsilon,R}$. Moreover the mixed constraint $\Phi(x,u) \in \Omega$ is independent of $y$. Hence the WBCQ  in Proposition \ref{Propnew}
and the calmness condition hold.
 By  Proposition \ref{Propnew}, there exist an arc $(p,q)$ such that  the nontriviality condition $(\lambda_0, p(t),q(t))\not =0, \forall t\in [t_0,t_1]$ holds, the  transversality condition as in Theorem \ref{thm4.2new}
 holds, the Euler adjoint inclusion  in the form
\begin{eqnarray}
\lefteqn{(\dot{p}(t),\dot{q}(t), 0) \in}\nonumber\\
&&\partial^C \{\langle -p(t), \phi\rangle+\lambda_0F\} ( x_*(t),y_*(t),u_*(t))+\{(0,0)\}\times N^C_{U}(u_*(t))\nonumber \\
&& +co\{ \partial \langle  \lambda,  \Phi\rangle ( x_*(t),u_*(t)): \lambda\in N_{\Omega}(\Phi(x_*(t),u_*(t))\}\mbox{ a.e.}\label{oldEuler}
\end{eqnarray}
holds,
and the Weierstrass condition of radius $R_\rho(\cdot)$  holds in the form that for almost every $t$:
\begin{eqnarray}
&&(x_*(t),u)\in M(0), \, \,  |(\phi(x_*(t),u)-\phi(x_*(t),u_*(t)), \rho (u-u_*(t)))|< \rho R(t)\Longrightarrow\nonumber\\
&&\langle p(t),\phi(x_*(t),u) \rangle-\lambda_{0}F(x_*(t),u)\leq \langle p(t), \phi (x_*(t),u_*(t)) \rangle -\lambda_{0}F(x_*(t),u_*(t)). \nonumber \\
&& \qquad \qquad \label{Weierstrass}
\end{eqnarray}
Because $\phi,F,\Phi$ are independent of $y$, it follows from (\ref{oldEuler}) that $\dot{q}(t)\equiv0$ a.e.. Together with $q(t_1)=0$ implies that ${q}(t)\equiv 0$.  Hence (\ref{oldEuler}) implies
the Euler adjoint inclusion (\ref{EulerNnew}) and the nontriviality condition as in Theorem \ref{thm4.2new}.

Since $\tilde{C}_*^{\varepsilon,R}$ is compact, the set
$$C:=cl\left \{\cup_{t\in[t_0,t_1]} (x_*(t),u)\in M(0):\ |u-u_*(t)|\leq R(t) \right\}   $$
is compact as well. Since $\phi(x,u)$ is locally Lipschitz continuous and $C$ is compact, one can find a positive constant $k^\phi_u$ such that
$$ |\phi(x_*(t), u_1)-\phi(x_*(t),u_2)|\leq k^\phi_u|u_1-u_2| \quad \forall (x_*(t), u_1), (x_*(t), u_2)\in C.$$ Let $(x_*(t), u)\in M(0), |u-u_*(t)|<R(t)$. Then $(x_*(t), u), (x_*(t), u_*(t))\in C$ and hence
$$|(\phi(x_*(t),u)-\phi(x_*(t),u_*(t)), \rho (u-u_*(t)))|\leq  \max\{k^{\phi}_{u},\rho\}|u-u_*(t)|<\max\{k^{\phi}_{u},\rho\}R(t).$$ Take a special $\rho >k_u^\phi$. Then $\max\{k^{\phi}_{u},\rho\}=\rho$ and hence
 (\ref{Weierstrass}) implies that the Weierstrass condition in Theorem \ref{thm4.2new} holds.
Moreover as discussed in \cite[Remark 3.1]{anli}, $\lambda_0$ can be chosen as $1$ in the case of free end point.
  \hfill $\Box$
\end{document}